\documentclass[10pt]{article}
\usepackage[T1]{fontenc}
\usepackage{amsmath,amsfonts,amsthm,mathrsfs,amssymb}
\usepackage{cite}
\usepackage{multirow}
\usepackage{graphicx,float}
\usepackage{subfigure}
\usepackage{placeins}
\usepackage{color}
\usepackage{indentfirst}
\usepackage[bookmarks=true]{hyperref}
\definecolor{r1}{rgb}{0.000,0.000,0.000}
\numberwithin{equation}{section}
\newcommand{\R}{{\mathbb R}}

\newcommand{\C}{{\mathbb C}}

\newcommand{\be}{\begin{eqnarray}}
\newcommand{\ben}{\begin{eqnarray*}}
\newcommand{\en}{\end{eqnarray}}
\newcommand{\enn}{\end{eqnarray*}}

\newcommand{\pa}{\partial}

\newcommand{\ov}{\overline}

\newtheorem{theorem}{Theorem}[section]
\newtheorem{lemma}[theorem]{Lemma}

\newtheorem{remark}[theorem]{Remark}

\definecolor{rot}{rgb}{1.000,0.000,0.000}
\begin{document}
\renewcommand{\theequation}{\arabic{section}.\arabic{equation}}
\begin{titlepage}
\title{A highly accurate perfectly-matched-layer boundary integral equation solver for acoustic layered-medium problems}
\author{Wangtao Lu\thanks{School of Mathematical Sciences, Zhejiang University, Hangzhou, Zhejiang 310027, China. Email: {\tt wangtaolu@zju.edu.cn}}\;,
Liwei Xu\thanks{School of Mathematical Sciences, University of Electronic Science and Technology of China, Chengdu, Sichuan 611731, China. Email: {\tt xul@uestc.edu.cn}}\;,
Tao Yin\thanks{LSEC, Institute of Computational Mathematics and Scientific/Engineering Computing, Academy of Mathematics and Systems Science, Chinese Academy of Sciences, Beijing 100190, China. Email:{\tt yintao@lsec.cc.ac.cn}}\;,
Lu Zhang\thanks{School of Mathematical Sciences, University of Electronic Science and Technology of China, Chengdu, Sichuan 611731, China. Email: {\tt luzhang@std.uestc.edu.cn}}}
\date{}
\end{titlepage}
\maketitle

\begin{abstract}
Based on the perfectly matched layer (PML) technique, this paper develops a high-accuracy boundary integral equation (BIE) solver for acoustic scattering problems in locally defected layered media in both two and three dimensions. The original scattering problem is truncated onto a bounded domain by the PML. Assuming the vanishing of the scattered field on the PML boundary, we derive BIEs on local defects only in terms of using PML-transformed free-space Green's function, and the four standard integral operators: single-layer, double-layer, transpose of double-layer, and hyper-singular boundary integral operators. \textcolor{r1}{The hyper-singular integral operator is transformed into a combination of weakly-singular integral operators and tangential derivatives.} We develop a high-order Chebyshev-based rectangular-polar singular-integration solver to discretize all weakly-singular integrals. Numerical experiments for both two- and three-dimensional problems are carried out to demonstrate the accuracy and efficiency of the proposed solver.

\end{abstract}

\section{Introduction}
\label{sec1}
The scattering problems in a locally perturbed half-space have attracted much attention of engineers and mathematicians for many years, which arise from various applications, such as modeling acoustic and electromagnetic wave propagation over outdoor ground and sea surface, optical scattering from the surface of materials in near-field optics or nano-optics, underwater detection~\cite{C95}. For unbounded exterior domain \textcolor{r1}{problems}, the boundary integral equation (BIE) method discretizes boundaries only, reducing the problem dimension by one, and automatically satisfies the outgoing radiation condition. Therefore, it has been widely used for various scattering problems~\cite{BXY17,BXY191,BET12,BY20,CBB91,CN02,DM97,MC96,TCPS02,YHX17,ZXY21,ZXY22}.

In the literature, existing BIE solvers for layered-medium scattering problems can be classified into two approaches. First approach uses the background Green's functions~\cite{C02,PGM00,PB14} to build up the governing BIEs. They are defined only on the bounded perturbed part of the scattering surface so that no truncation is required further. Nevertheless, background Green's functions involve sophisticated Sommerfeld integrals and how to effectively evaluate them becomes the key ingredient of this approach. The second approach uses the free-space Green functions instead and then the resulting BIEs are established on the unbounded scattering surface. Special treatments are required to truncate the unbounded surface. They include the approximate truncation method~\cite{MC01,SS11}, the taper function method~\cite{MSS14,SSS08,ZLSC05}, and the windowed Green function method~\cite{BD14,BGP17,BLPT16,BP17,BY21}. It is worth noting that among these methods, the windowed Green function method is the only high-accuracy method and enjoys a super-algebraically convergence rate. It introduces a correction that smoothly merges the unknown functions in the original integral equations with the corresponding scattered solutions for the unperturbed flat surface, thus providing uniformly fast convergence for all incident angles as the support of the windowing function grows. However, for point-source incidences, this method needs to deal with multiple spherical-wave incidences from the expansion of the incident fields.

In a recent work~\cite{LLQ18}, a PML-based BIE method was proposed for solving a two-layer wave scattering problem in two-dimensions. Similar to the second approach mentioned above, it builds up governing BIEs on an unbounded scattering surface but uses PML-transformed free-space Green's functions instead. Due to the outgoing behavior of the scattered field, waves along the unbounded surface decay exponentially inside the PML. Directly truncating the unbounded surface produces numerical solutions converging exponentially fast with the PML absorbing powers. This high-accuracy method has so far been successfully extended to more complicated structures, such as step-like scattering surfaces~\cite{L21}, anisotropic media~\cite{GL22}, and locally perturbed periodic structures~\cite{YHLR22}.

The present paper develops a new PML-based BIE solver, significantly improving the original one from two aspects. First, the original PML-based BIE method uses only the single-layer and double-layer operators, leading to a first-kind Fredholm system. The new method uses two extra operators: the transpose of double-layer operator and the hyper-singular operator to establish second-kind Fredholm systems. Second, Alpert's high-order quadrature rule~\cite{A99}, utilized by~\cite{LLQ18} in discretizing the integral operators, has not been extended to three dimensions yet. We transform the hyper-singular integral operator into weakly-singular operators and their tangential derivatives. Then, a Chebyshev-based rectangular-polar integral solver is utilized to discretize all weakly-singular integral operators for both two and three dimensions. The tangential derivatives are then obtained by directly differentiating the corresponding truncated Chebyshev expansions. Numerical examples demonstrate that, typically, we only need to set the PML thickness to twice the wavelength to obtain high accuracy and fast convergence for two- and three-dimensional scattering problems.

This paper is organized as follows. In Section~\ref{sec2}, we describe the acoustic half-space scattering problems under Dirichlet and Neumann boundary conditions, and present corresponding BIEs based on the free-space Green function. Section~\ref{sec3} presents the PML-based BIEs for solving Dirichlet and Neumann problems. In Section~\ref{sec4.1}, we derive the regularized formulations for two- and three-dimensional hyper-singular BIOs. A high order discretization method for the BIOs is proposed in Section~\ref{sec4.2}. \textcolor{r1}{Section~\ref{sec4.3} extends the new PML-based BIE solver to layered-medium scattering problems}. A variety of numerical examples in two and three dimensions are presented in Section~\ref{sec5} to illustrate the performance of our method.

\section{Preliminaries}
\label{sec2}
\subsection{Half-space scattering problems}
\label{sec2.1}
This section is devoted to an efficient and highly accurate boundary integral solver for the acoustic half-space problems. To simplify the presentation, we consider the acoustic scattering by a combination of
impenetrable bounded obstacles and an infinite flat surface which may consist of some local defects. An extension of the solver to more complicated structures shall be discussed in Section~\ref{sec4.3}.

\begin{figure}[htb]
	\centering
	\begin{tabular}{c}
		\includegraphics[scale=0.3]{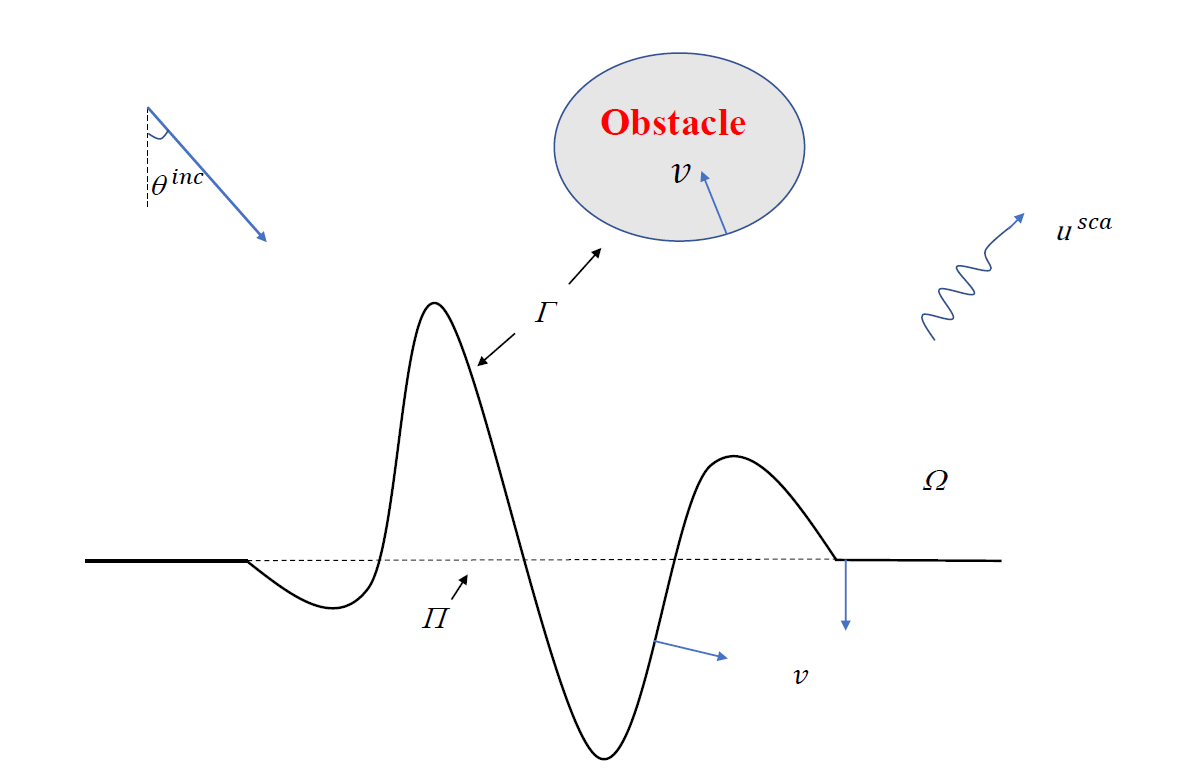}
	\end{tabular}
	\caption{Problems of scattering by a locally perturbed half-space.}
	\label{obstacle}
\end{figure}

As shown in Fig.~\ref{obstacle}, $\Omega\in\R^d$, $d=2,3$ denotes an unbounded connect open domain such that there exists constants $f_{-}<f_{+}$ with
\ben
U_{f_{+}} \subset \Omega \subset U_{f_{-}}, \quad U_{f_{\pm}}:=\left\{x=(x_1,...,x_d)\in\R^d: x_d>f_{\pm} \right\}.
\enn
The boundary $\Gamma:=\partial\Omega$ contains two parts: the unbounded flat surface
\ben
\Pi:=\left\{ x\in \R^d:x_d=0 \right\}
\enn
with some local defects and the boundary $\partial U_0$ of a bounded obstacle $U_0$. Let $u^{\mathrm{inc}}$ be a time-harmonic incident field which is a plane pressure wave:
\be
\label{PPW}
u^{\mathrm{inc}}(x)=\exp(ikx\cdot d^{\mathrm{inc}}),\quad x\in\R^d,
\en
or a point source located at $z\in\Omega$:
\be
\label{PS}
u^{\mathrm{inc}}(x)=G(x,z):=\begin{cases}
\frac{i}{4}H^{(1)}_0(k|x-z|), & d=2,\cr
\frac{\exp(ik|x-z|)}{4\pi|x-z|},& d=3,
\end{cases} \quad x\in\R^d,x\ne z,
\en
where $k>0$ denotes the angular frequency, $H^{(1)}_0$ is the first-kind Hankel function of order zero and the incident direction
\ben
d^{\mathrm{inc}}=\begin{pmatrix}
		\sin \theta^{\mathrm{inc}}\\
		-\cos \theta^{\mathrm{inc}}
\end{pmatrix} \quad \mbox{in} \quad \mbox{2D} \quad \mbox{and} \quad d^{\mathrm{inc}}=\begin{pmatrix}
		\sin \theta^{\mathrm{inc}}\cos\phi^{\mathrm{inc}}\\
		\sin \theta^{\mathrm{inc}}\sin\phi^{\mathrm{inc}}\\
		-\cos \theta^{\mathrm{inc}}
\end{pmatrix} \quad \mbox{in} \quad \mbox{3D},
\enn
with $\theta^{\mathrm{inc}}$ and $(\theta^{\mathrm{inc}},\phi^{\mathrm{inc}})$ being the incident angle and angle pair, respectively.

The scattered field $u^{\mathrm{sca}}$ satisfies the Helmholtz equation
\be
\label{HME}
&&\Delta u^{\mathrm{sca}}+k^2u^{\mathrm{sca}}=0.
\en
For simplicity, on the boundary $\Gamma$, we impose the Dirichlet boundary condition
\be
\label{ODBC}
&&u^{\mathrm{sca}}=f (:=-u^{\mathrm{inc}}-u_{\mathrm{ref}}^{\mathrm{sca}}),
\en
or the Neumann boundary condition
\be
\label{ONPC}
&&\partial_\nu u^{\mathrm{sca}}=g (:=-\partial_\nu u^{\mathrm{inc}}-\partial_\nu u_{\mathrm{ref}}^{\mathrm{sca}}).
\en
Here $\nu$ is the unit outward normal and $\partial_\nu:=\nu\cdot\nabla$ denotes the normal derivative. In addition, $u_{\mathrm{ref}}^{\mathrm{sca}}=0$ for incidence of point source. For the case of plane-wave incidence, $u_{\mathrm{ref}}^{\mathrm{sca}}$ represents the reflected field resulting from the scattering of the plane wave $u^{\mathrm{inc}}$ by the flat surface $\Pi$. In particular, it is easy to deduce that
\ben
u_{\mathrm{ref}}^{\mathrm{sca}}(x):=\begin{cases}
		-\exp(ikx'\cdot d^{\mathrm{inc}}), &\mbox{Dirichlet case}, \cr
	    \exp(ikx'\cdot d^{\mathrm{inc}}), &\mbox{Neumann case},
\end{cases}
\enn
where $x'=(x_1,\cdots,-x_d)$ denotes the imaging point of $x$ w.r.t. the flat surface $\Pi$.

At infinity, the scattered field satisfies the following half-space Sommerfeld radiation condition:
	\begin{equation}
		\label{eq:src}
		\lim_{r\to\infty}r^{(n-1)/2}(\partial_{r} - i k) u^{\rm sca} = 0,
	\end{equation}
uniformly holds in all directions for $x_3\geq 0$, where $r=|x|$; see the angular spectrum representation~\cite{DM97} and also the equivalent upward propagating radiation condition~\cite{AH05,CZ98}. \textcolor{r1}{For theories on direct and inverse rough-surface scattering problems, we refer readers to~\cite{BHY18,BL11,W87,ZZ13} for more details.}

\subsection{Boundary integral equations}
\label{sec2.2}
It follows from~\cite{DM97} that the scattered field $u^{\mathrm{sca}}$ admits the representation
\be
\label{SR}
u^{\mathrm{sca}}(x)=\int_\Gamma \left\{ G(x,y)\partial_{\nu_y} u^{\mathrm{sca}}(y)-\partial_{\nu_y} G(x,y)u^{\mathrm{sca}}(y) \right\}ds_y, \quad x\in \Omega.
\en
\textcolor{r1}{Letting $x\to\Gamma$ and applying the well-known jump conditions~\cite{CK98} lead to the BIEs on $\Gamma$}
\be
\label{BIED}
&&\left(\frac{1}{2}I +K\right)(u^{\mathrm{sca}})(x)=S(\partial_\nu u^{\mathrm{sca}})(x), \quad x\in\Gamma,\\
\label{BIEN}
&&\left(-\frac{1}{2}I +K'\right)(\partial_\nu u^{\mathrm{sca}})(x)=N(u^{\mathrm{sca}})(x), \quad x\in\Gamma,
\en
wherein
\be
\label{SO}
S(\varphi)(x)&:=&\int_\Gamma  G(x,y)\varphi (y)ds_y,\quad x\in\Gamma,\\
\label{KO}
K(\varphi)(x)&:=&\int_\Gamma \partial_{\nu_y}G(x,y) \varphi(y)ds_y,\quad x\in\Gamma,\\
\label{KTO}
K'(\varphi)(x) &:=&\int_\Gamma  \partial_{\nu_x} G(x,y)\varphi (y)ds_y,\quad x\in\Gamma,\\
\label{HO}
N(\varphi)(x) &:=&\int_\Gamma  \partial_{\nu_x}\partial_{\nu_y}G(x,y)\varphi (y)ds_y,\quad x\in\Gamma,
\en
denote, respectively, the single-layer, double-layer, transpose of double-layer, and hyper-singular BIOs. Especially, the hyper-singular operator $N$ is defined in the sense of Hadamard finite part~\cite{HW08}.

In light of the advantages of using second-kind Fredholm integral equations for solving large-scale problems, (\ref{BIED}) and (\ref{BIEN}) can be utilized for solving the Neumann and Dirichlet problems, respectively. However, the BIOs (\ref{SO})-(\ref{HO}) are defined on the whole unbounded surface $\Gamma$, so that we require appropriate truncation strategy for the numerical implementation.

A direct truncation of the infinite boundary $\Gamma$ in the BIEs (\ref{BIED})-(\ref{BIEN}) will lead to a large truncation error. In particular, a smooth windowing function is introduced in~\cite{BLPT16} for the truncation. As illustrated in that work, for plane-wave incidence, the direct windowing approach requires, for a given accuracy, increasingly large truncated domains as grazing incidence is approached. To overcome this poor performance, corrected formulations are proposed to uniform accuracy for all incident angles. However, for the case of point source, the incident field should be expressed by a superposition of plane incident waves, which means that a number of problems of plane incidence should be considered, and it is necessary to apply the Cauchy's theorem to deform the integration contour.

In~\cite{LLQ18}, a PML-based BIE method is proposed for solving the two-dimensional wave scattering problems in a layered medium. This method
truncates the considered infinite domain using the PML by directly imposing zero Dirichlet boundary condition on the PML boundary.  Alpert's hybrid Gauss-trapezoidal quadrature rule~\cite{A99} is utilized for the
numerical discretization of the two-dimensional single- and double-layer BIOs and high accuracy can be achieved for incidences of plane waves and point sources. Unfortunately, quadrature rules for surface integrals that are analogous to Alpert's are still absent so far. To tackle this difficulty, we develop a high-accuracy Chebyshev-based rectangular-polar integral solver that is applicable for both line and surface integrals respectively corresponding to the two- and three-dimensional problems, for the discretization of the PML-transformed versions of all four BIOs (\ref{SO})-(\ref{HO}) in the following.

\section{The truncated PML problems and boundary integral equations}
\label{sec3}

In this section, we briefly discuss the PML truncation strategy and introduce the reduced boundary integral equations for the truncated PML problems. The core idea of the PML is to construct artificial layers (of finite thickness) surrounding the physical bounded domain such that the outgoing waves $u^{\mathrm{sca}}$ and $\pa_\nu u^{\mathrm{sca}}$ decay rapidly in the PML region before reaching the PML outer boundary. Then the BIEs for the truncated PML problems can be derived.

\subsection{The PML stretching}
\label{sec3.1}

The PML method involves analytical stretching of the real spatial coordinates of the physical equations into the complex plane, along which the outward propagating waves must be attenuated in the absorbing layer. Following the coordinate stretching approach~\cite{CW94}, we introduce a complex change of spatial variable: $x\in \Omega \subset \R^d\mapsto \widetilde x\in \widetilde \Omega\subset \C^d$ defined as
\be
\label{CCS}
\widetilde x_i(x_i)=x_i+i\int_0^{x_i}\sigma_i(t)dt
\en
for $i=1,\cdots,d$, where we take $\sigma_i$
\be
\label{pos}
\sigma_i(t)=\sigma_i(-t),\quad \sigma_i(t)=0\;\;\mbox{for}\;\; \left| t \right|\le a_i,\quad \sigma_i(t)>0 \;\;\mbox{for}\;\; \left| t \right|> a_i,
\en
with $a_i>0$, $i=1,\cdots,d$. The domains with nonzero $\sigma_i$ are called the PML. For definiteness, throughout this paper we utilize the positive function $\sigma_i$~\cite{CK98,LH19,LLQ18}
\be
\label{sigma}
\sigma_i(x_i)=\left\{ {\begin{array}{ll}
		\frac{2Sf_1^P}{f_1^P+f_2^P}, & a_i\le x_i\le a_i+T_i,\\
		S,& x_i>a_i+T_i,\\
		0,&-a_i<x_i<a_i,\\
		\sigma_i(-x_i),&x\le -a_i,
\end{array}} \right.
\en
where $T_i>0$, $i=1,\cdots,d$ denote the thickness of the PML, $P$ is a positive integer,
\ben
f_i=\left(\frac{1}{2}-\frac{1}{P}\right)\overline x_i^3 +\frac{\overline x_i}{P}+\frac{1}{2},\quad f_2=1-f_1,\quad\overline x_i=\frac{x_i-\left(a_i+T_i\right)}{T_i}.
\enn
It can be seen that $\sigma_i$ maps $[a_i ,a_i+T_i]$ onto $[0 ,S]$ and its derivatives vanish at $x_i=\pm a_i$ up to order $P-1$. In addition, we choose $a_i>0$ such that the bounded Cartesian domain $B_a:=(-a_1, a_1)\times \cdots\times (-a_d, a_d)\subset \R^d$ encloses the bounded obstacle $U_0$ and the local defects on $\Gamma$. Then the infinite domain $\Omega$ can be truncated onto a bounded domain $\Omega^b$, which also leads to the corresponding truncated interface $\Gamma^b$ of $\Gamma$, using the box $B_{a,T}=(-a_1-T_1, a_1+T_1)\times\cdots\times (-a_d-T_d, a_d+T_d)$, i.e., $\Omega^b=\Omega\cap B_{a,T}$ and $\Gamma^b=\Gamma\cap B_{a,T}$. In addition, we denote $\Gamma^+=\partial\Omega^b\backslash\Gamma^b$.

\begin{figure}[htb]
	\centering
	\begin{tabular}{ccc}
		\includegraphics[scale=0.15]{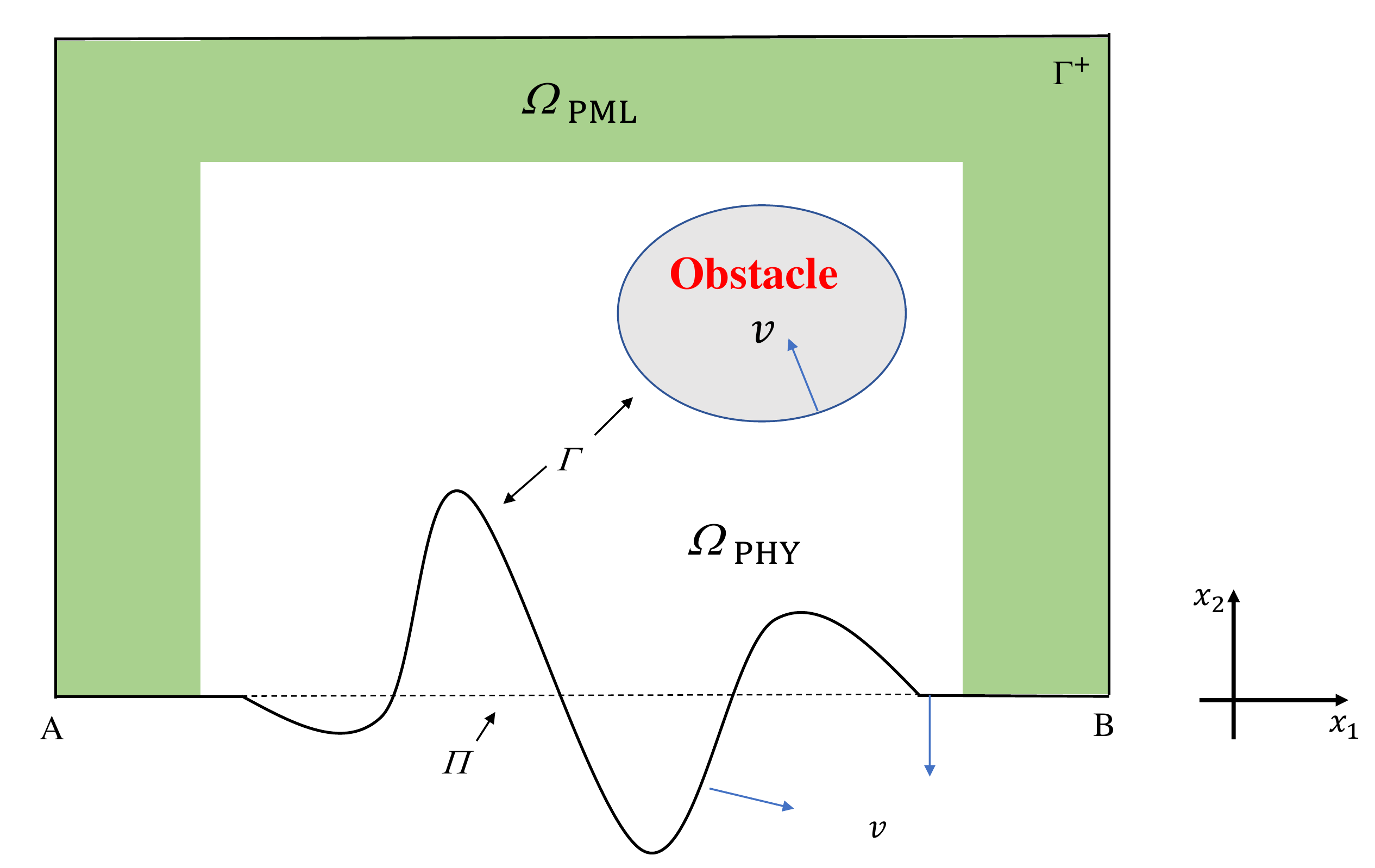} &
		\includegraphics[scale=0.15]{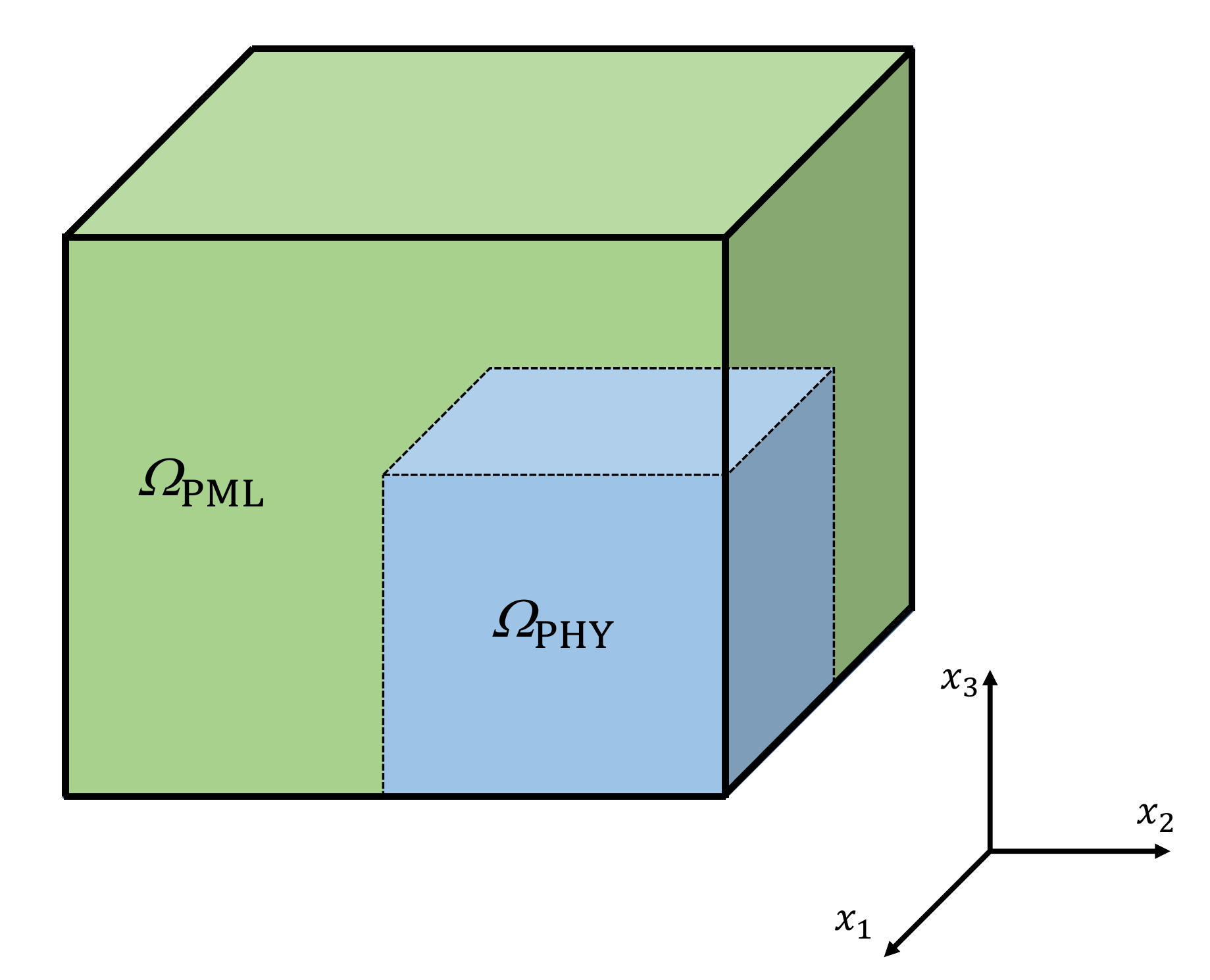} &
		\includegraphics[scale=0.18]{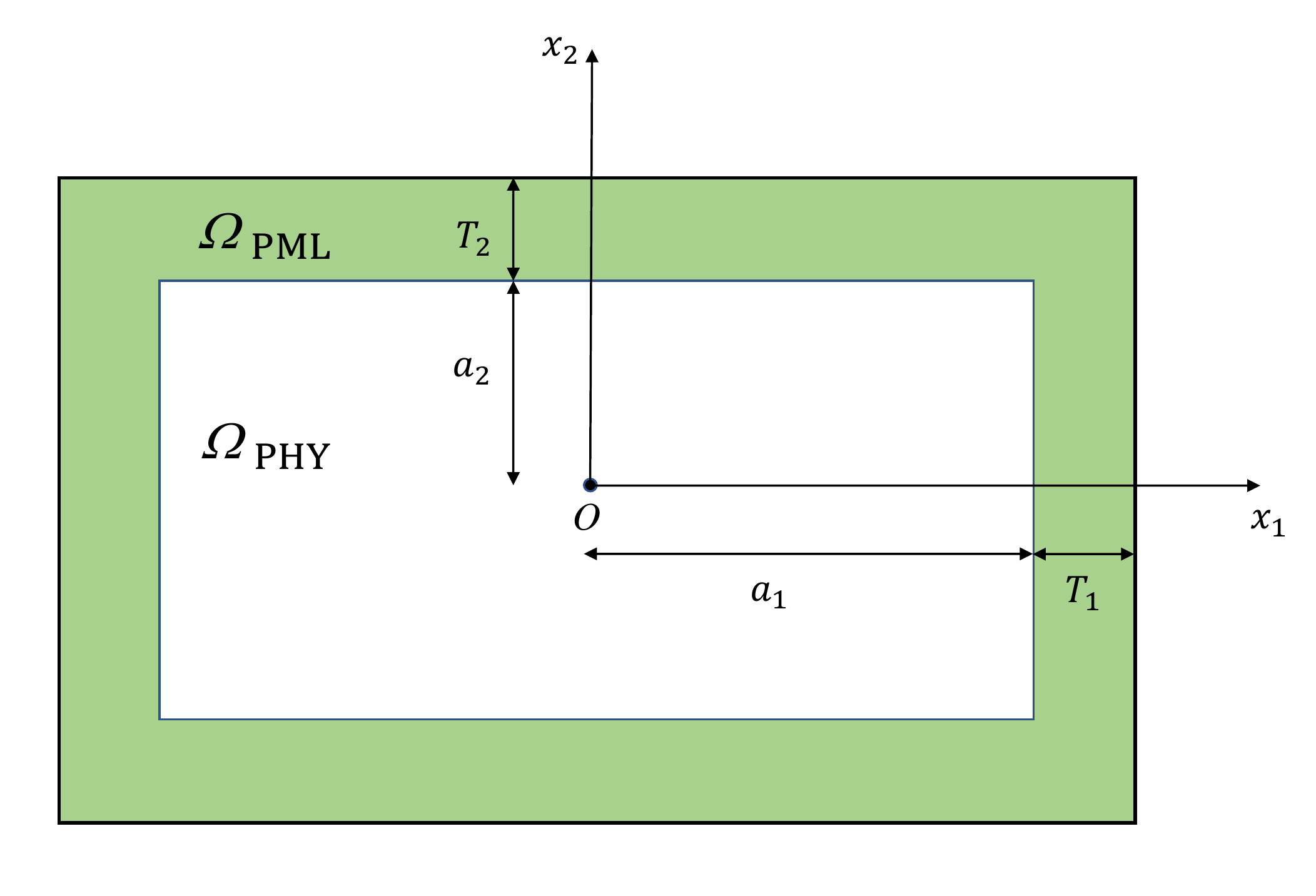} \\
		(a)&(b)&(c)
	\end{tabular}
	\caption{The PML-truncated domain $\Omega^b$, the physical bounded domain $\Omega_{\rm PHY}=B_a\cap \Omega$ and the PML region $\Omega_{{\rm PML}}=B_{a,T}\backslash\ov{B_a}$: (a) the PML truncation in two dimensions; (b) the quarter model of PML truncation in three dimensions; (c) the vertical view of PML truncation on the bottom infinite surface in three dimensions.}
	\label{PML23D}
\end{figure}

\subsection{The truncated PML problems}
\label{sec3.2}

Relying on the complex coordinate map (\ref{CCS}), the Helmholtz equation (\ref{HME}) on $\Omega^b$ can be transformed into the following form
\be
\label{CHE}
\widetilde \Delta u^{\mathrm{sca}}(\widetilde x)+k^2u^{\mathrm{sca}}(\widetilde x)=0 ,\quad \widetilde x\in\widetilde \Omega^b,
\en
where $\widetilde\Omega^b=\left\{ \widetilde x(x) | x\in \Omega^b \right\}$ and $\widetilde \Delta=\sum\limits_{i=1}^d \partial^2_{\widetilde x_i} $. Then from the Green's representation theorem~\cite{LS01}, the solution of (\ref{CHE}) can be represented by
\be
u^{\mathrm{sca}}(\widetilde x)=\int_{\Gamma^b}\left\{ G(\widetilde x,y)\partial_{\nu_y} u^{\mathrm{sca}}(y)-\partial_{\nu_y} G(\widetilde x,y)u^{\mathrm{sca}}(y) \right\}ds_y,\quad \widetilde x\in\widetilde \Omega^b.
\en
Defining the complex function $\widetilde u^{\mathrm{sca}}(x)=u^{\mathrm{sca}}(\widetilde x)$ in $\Omega^b$, we can rewrite the equation (\ref{CHE}) as
\be
\label{RCHE}
\nabla\cdot \left(\mathbf A\nabla \widetilde u^{\mathrm{sca}}\right)+k^2 J\widetilde u^{\mathrm{sca}}=0,\quad x\in\Omega^b,
\en
where
\ben
&&\alpha_i(x_i)=1+i\sigma_i(x_i),\quad i=1,...,d, \\
&&\mathbf A=\left\{ {\begin{array}{ll}
		\mbox{diag}\left\{\alpha_2/\alpha_1,\alpha_1/\alpha_2\right\},&d=2,\\
		\mbox{diag}\left\{\alpha_2\alpha_3/\alpha_1,\alpha_1\alpha_3/\alpha_2,\alpha_1\alpha_2/\alpha_3\right\},&d=3,
\end{array}} \right.\\
&& J=\left\{ {\begin{array}{ll}
		\alpha_1\alpha_2,&d=2,\\
		\alpha_1\alpha_2\alpha_3,&d=3.
\end{array}} \right.
\enn

On $\Gamma^b$, the original Dirichlet and Neumann boundary conditions on $\Gamma$ can be reduced to
\ben
&&\widetilde u^{\mathrm{sca}}=\widetilde f(x)\quad (:=f(\widetilde x)),
\enn
and
\ben
&&\widetilde\partial_{\nu}\widetilde u^{\mathrm{sca}}=\widetilde g(x)\quad (:=g(\widetilde x)),
\enn
respectively, where $\widetilde\partial_{\nu}=(\mathbf A^\top \nu)\cdot\nabla$. Sommerfeld's radiation condition (\ref{eq:src}) implies that $u^{\mathrm{sca}}$ is outgoing so that the complexified fields $\widetilde u^{\mathrm{sca}}$ and $\widetilde\partial_{\nu}\widetilde u^{\mathrm{sca}}$ decay exponentially as $\left|x\right|\to\infty$. Thus, it is highly accurate to directly assume that $\widetilde u^{\mathrm{sca}}=0$ and $\widetilde\partial_{\nu}\widetilde u^{\mathrm{sca}}=0$ on $\Gamma^+$.

\subsection{Boundary integral equations}
\label{sec3.3}

As shown in~\cite{LS01}, the fundamental solution of (\ref{RCHE}), called PML-transformed free-space Green's function, takes the form
\ben
\widetilde G(x,y)=G(\widetilde x,\widetilde y)=\left\{ {\begin{array}{ll}
		\frac{i}{4}H^{(1)}_0(k\rho(\widetilde x,\widetilde y)), & d=2,\\
		\frac{{\rm exp}\left(ik\rho(\widetilde x,\widetilde y)\right)}{4\pi\rho(\widetilde x,\widetilde y)},& d=3,
\end{array}} \right.
\enn
where $\rho$ is the complex distance function given by
\be
\label{rho}
\rho(\widetilde x,\widetilde y)=\left(\sum\limits_{j=1}^d \left(\widetilde x_j-\widetilde y_j\right)^2 \right)^{1/2}
\en
and the half-power operator $z^{1/2}$ is chosen to be the branch of $\sqrt z$ with nonnegative real part for $z\in\mathcal \C\backslash \left( -\infty,  \right.\left. 0 \right]$ such that $\mathrm{arg}\sqrt{z}\in (-\frac{\pi}{2},\frac{\pi}{2}]$.

According to~\cite{LLQ18} and noting that we have assumed that $\widetilde u^{\mathrm{sca}}\approx0$ and $\partial_{\nu_c}\widetilde u^{\mathrm{sca}}\approx0$ on $\Gamma^+$, the solution of (\ref{RCHE}) can be approximately represented in the form
\be
\label{DR}
\widetilde u^{\mathrm{sca}}(x)=\int_{\Gamma^b} \left\{ \widetilde G(x,y)\widetilde\partial_{\nu_y}\widetilde u^{\mathrm{sca}}(y)-\widetilde\partial_{\nu_y}\widetilde G(x,y)\widetilde u^{\mathrm{sca}}(y) \right\}ds_y, \quad x\in \Omega^b.
\en
Taking the limits as $x\to\Gamma^b$ and applying the jump conditions, we obtain the corresponding BIEs on $\Gamma^b$
\be
\label{PMLBIED}
&&\left(\frac{1}{2}I+\widetilde K\right)(\widetilde u^{\mathrm{sca}})(x)=\widetilde S(\widetilde\partial_{\nu} \widetilde u^{\mathrm{sca}})(x), \quad x\in\Gamma^b,\\
\label{PMLBIEN}
&&\left(-\frac{1}{2}I+\widetilde K'\right)(\widetilde\partial_{\nu} \widetilde u^{\mathrm{sca}})(x)=\widetilde N(\widetilde u^{\mathrm{sca}})(x), \quad x\in\Gamma^b,
\en
where the boundary integral operators $\widetilde S$, $\widetilde K$, $\widetilde K'$ and $\widetilde N$ are defined by (\ref{SO}), (\ref{KO}), (\ref{KTO}) and (\ref{HO})
with $G$ replaced by $\widetilde G$, $\partial_\nu$ replaced by $\widetilde\partial_{\nu}$ and $\Gamma$ replaced by $\Gamma^b$, respectively.

Analogous to the discussion in Section~\ref{sec2.2}, we use the BIEs
(\ref{PMLBIED}) and (\ref{PMLBIEN}) to solve the corresponding Neumann and
Dirichlet problems, respectively. The given Dirichlet and Neumann data on $\Gamma$ imply the following two BIEs:
\be
\label{PMLBIE1}
\left(\frac{1}{2}I+\widetilde K\right)(\widetilde u^{\mathrm{sca}})(x)=\widetilde S(\widetilde g)(x), \quad x\in\Gamma^b,
\en
and
\be
\label{PMLBIE2}
\left(-\frac{1}{2}I+\widetilde K'\right)(\widetilde\partial_{\nu} \widetilde u^{\mathrm{sca}})(x)=\widetilde N(\widetilde f)(x), \quad x\in\Gamma^b.
\en

\section{Numerical implementation}
\label{sec4}

This section will discuss the numerical discretization of the boundary integral operators $\widetilde S$, $\widetilde K$, $\widetilde K'$ and $\widetilde N$ by means of the Chebyshev-based rectangular-polar integral solver~\cite{BG20,BY20}. Before that, regularized formulations will be derived in the following subsection to treat the hyper-singular operator $\widetilde N$.

\subsection{Regularization of the hyper-singular operator}
\label{sec4.1}

As shown in the following lemma, by analogy to Maue's identity~\cite{K95,M49}, the hyper-singular operator $\widetilde N$ can be reformulated as a combination of weakly-singular integral operators and tangential derivatives.
\begin{lemma}
\label{HSreg}
Assuming that $\varphi=0$ on $\partial\Gamma^b$. In two dimensions, the hyper-singular operator $\widetilde N$ can be expressed in the form
	\ben
	\label{N2D}
	\widetilde{N}(\varphi)(x)=\frac{d}{ds_x}\int_{\Gamma^b}\widetilde{G}(x,y)\frac{d\varphi}{ds_y}ds_y+k^2\int_{\Gamma^b}\nu_x^\top \mathbf A_1(x,y)\nu_y\widetilde{G}(x,y)\varphi(y)ds_y,
	\enn
	where $\mathbf  A_1(x,y)=\mbox{diag}\left\{\alpha_2(x_2)\alpha_2(y_2),\alpha_1(x_1)\alpha_1(y_1) \right\}$ and $\frac{d}{ds}=\nu^{\perp}\cdot\nabla$ denotes the tangential derivative. In three dimensions, we have
	\ben
	\label{N3D}
	\widetilde N(\varphi)(x)= &&\int_{\Gamma^b}  \left(\mathbf A_2(x,y)\left(\nu_x\times \nabla_x\right)\widetilde G(x,y)\right)\cdot\left(\nu_y\times\nabla_y\right)\varphi(y)ds_y\\
&& +k^2\int_{\Gamma^b}\nu_x^\top\mathbf A_3(x,y) \nu_y\widetilde G(x,y)\varphi(y)ds_y,
	\enn
	where
	\ben
 &&\mathbf A_2(x,y)=\mbox{diag}\left\{\alpha_1(x_1)\alpha_1(y_1),\alpha_2(x_2)\alpha_2(y_2),\alpha_3(x_3)\alpha_3(y_3) \right\}, \\
 &&\mathbf A_3(x,y)=\mbox{diag}\left\{\frac{\alpha(x,y)}{\alpha_1(x_1)\alpha_1(y_1)},\frac{\alpha(x,y)}{\alpha_2(x_2)\alpha_2(y_2)}, \frac{\alpha(x,y)}{\alpha_3(x_3)\alpha_3(y_3)}\right\},
 \enn
 with $\alpha(x,y)=\mathop \Pi \limits_{i = 1}^3\alpha_i(x_i)\alpha_i(y_i)$.
\end{lemma}
\begin{proof}
Here we only give the proof of three-dimensional case and the proof of two-dimensional case can be carried out analogously. The hyper-singular operator $\widetilde N$ is given by
\ben
&&\widetilde N(\varphi)(x)=\widetilde\partial_{\nu_x}\int_{\Gamma^b} \widetilde\partial_{\nu_y}\widetilde G(x,y)\varphi(y)ds_y,\quad x\in\Gamma^b.
\enn
Noting that
\ben
&&\widetilde\partial_{\nu_y}\widetilde G(x,y)=\left[\alpha_2(y_2)\alpha_3(y_3)\nu_y^1\partial_{\widetilde {y}_1}+\alpha_1(y_1)\alpha_3(y_3)\nu_y^2\partial_{\widetilde {y}_2}+\alpha_1(y_1)\alpha_2(y_2)\nu_y^3\partial_{\widetilde {y}_3}\right]\widetilde G(x,y),
\enn
and
\ben
\nabla\cdot \left(\mathbf A\nabla \widetilde G(x,y)\right)+k^2 J\widetilde G(x,y)=0,\quad x\ne y,
\enn
it can be derived that
\ben
&&\mathbf A\nabla_x\widetilde\partial_{\nu_y}\widetilde G(x,y)\\
&&=\mathbf C(x,y)\left(\nu_y\times\nabla_y\right)\widetilde  G(x,y)+k^2\begin{bmatrix}
	\alpha_2(x_2)\alpha_2(y_2)\alpha_3(x_3)\alpha_3(y_3)\nu_y^1\\
	\alpha_1(x_1)\alpha_1(y_1)\alpha_3(x_3)\alpha_3(y_3)\nu_y^2\\
	\alpha_1(x_1)\alpha_1(y_1)\alpha_2(x_2)\alpha_2(y_2)\nu_y^3
\end{bmatrix}\widetilde  G(x,y).
\enn
where
\ben
&&\mathbf C(x,y)=\\
&&\begin{bmatrix}
	0&\alpha_2(x_2)\alpha_3(x_3)\alpha_2(y_2)\partial_{\widetilde {x}_3}&-\alpha_2(x_2)\alpha_3(x_3)\alpha_3(y_3)\partial_{\widetilde {x}_2}\\
	-\alpha_1(x_1)\alpha_3(x_3)\alpha_1(y_1)\partial_{\widetilde {x}_3}&0&\alpha_1(x_1)\alpha_3(x_3)\alpha_3(y_3)\partial_{\widetilde {x}_1}\\
	\alpha_1(x_1)\alpha_2(x_2)\alpha_1(y_1)\partial_{\widetilde {x}_2}&-\alpha_1(x_1)\alpha_2(x_2)\alpha_2(y_2)\partial_{\widetilde {x}_1}&0
\end{bmatrix}.
\enn
Then the Stokes formula yields
\ben
N(\varphi)(x)=&&\int_{\Gamma^b} \begin{bmatrix}
	\alpha_1(x_1)\alpha_1(y_1)\left(\nu_x^2\partial_{x_3}-\nu_x^3\partial_{x_2}\right)\\
	\alpha_2(x_2)\alpha_2(y_2)\left(\nu_x^3\partial_{x_1}-\nu_x^1\partial_{x_3}\right)\\
	\alpha_3(x_3)\alpha_3(y_3)\left(\nu_x^1\partial_{x_2}-\nu_x^2\partial_{x_1}\right)
\end{bmatrix}^\top\widetilde G(x,y)\left(\nu_y\times\nabla_y\right)\varphi(y)ds_y\\
&&+k^2\int_{\Gamma^b}\nu_x^\top\begin{bmatrix}
	\alpha_2(x_2)\alpha_2(y_2)\alpha_3(x_3)\alpha_3(y_3)\nu_y^1\\
	\alpha_1(x_1)\alpha_1(y_1)\alpha_3(x_3)\alpha_3(y_3)\nu_y^2\\
	\alpha_1(x_1)\alpha_1(y_1)\alpha_2(x_2)\alpha_2(y_2)\nu_y^3
\end{bmatrix}\widetilde G(x,y)\varphi(y)ds_y\\
=&&\int_{\Gamma^b} \left(\mathbf A_2(x,y)\left(\nu_x\times \nabla_x\right)\widetilde G(x,y)\right)\cdot\left(\nu_y\times\nabla_y\right)\varphi(y)ds_y\\
&& +k^2\int_{\Gamma^b}\nu_x^\top\mathbf A_3(x,y) \nu_y\widetilde G(x,y)\varphi(y)ds_y,
\enn
and this completes the proof.
\end{proof}

\begin{remark}
For the numerical discretization of (\ref{PMLBIE2}), it requires to treat the term $\widetilde N(\widetilde f)$.
\begin{itemize}
\item For the case of a plane incident wave, it is known that $f=0$ on $\Pi$ and thus, $\widetilde f=0$ on $\partial\Gamma^b$. Hence, the regularized formulations shown in Lemma~\ref{HSreg} holds exactly.
\item For the case of a point source, $f=-G(x,z)$ for $x\in\Gamma$. Noting that $G(x,z)$ is an outgoing wave, the complex coordinate transform indicates that $G(\widetilde x,z)$ decays exponentially as $|x|\rightarrow\infty$. Thus, we have $\widetilde f\approx 0$ on $\partial\Gamma^b$. Then the regularized formulations shown in Lemma~\ref{HSreg} can be viewed as an approximation of $\widetilde N(\widetilde f)$.
\end{itemize}
\end{remark}

\subsection{Chebyshev-based rectangular-polar solver}
\label{sec4.2}

Relying on the regularization of the hyper-singular operator proposed in the previous subsection, the discretizations of the BIOs $\widetilde S$, $\widetilde K$, $\widetilde K'$ and $\widetilde N$ can be degenerated into the discretization of
\begin{itemize}
\item[(i)] the integral operators of the form
\be
\label{H}
\mathcal H(\varphi)(x)=\int_{\Gamma^b}H(x,y)\varphi(y)ds_y,\quad x\in\Gamma^b,
\en
in which the kernel $H(x,y)$ is piece-wise weakly-singular;
\item[(ii)] the tangential derivative operators $\frac{d}{ds}$ and $\nu\times\nabla$.
\end{itemize}

In this work, the Chebyshev-based rectangular-polar solver and the Chebyshev-based differentiation algorithm proposed in~\cite{BG20,BY20} will be applied for numerical evaluations of (i) and (ii), respectively. For simplicity, we only give a brief description for the three-dimensional case. For the two-dimensional case, we refer to~\cite{BY21}.

Let the surface $\Gamma^b$ be partitioned into a set of non-overlapping parameterized patches $\Gamma_q$, $q=1,...,M$ as
\ben
\Gamma^b=\bigcup_{q=1}^M \Gamma_q, \quad \Gamma_q=\left\{ x=\textbf{r}^q(u,v):\left[-1,1\right]^2\to \R^3 \right\}.
\enn
Then the integrals $\mathcal{H}\varphi(x)$ over $\Gamma^b$ can be split into the sum of integrals over each of the $M$ patches,
\ben
\mathcal{H}(\varphi)(x)=\sum\limits_{q=1}^{M} \mathcal{H}_q(x),\quad \mathcal{H}_q(x):=\int_{\Gamma_q}H(x,y)\varphi(y)ds_y,\quad x\in \Gamma^b.
\enn
Denoting by $u_j,v_j\in[-1,1]$ ($j=0,\cdots,N-1$) the $N$ Chebyshev points
\ben
u_j=\cos\left(\frac{2j+1}{2N}\pi\right),\quad v_j=\cos\left(\frac{2j+1}{2N}\pi\right),\quad j=0,\cdots,N-1,
\enn
we can define the discretization points on each $\Gamma_q$ by
\ben
x_{ij}^q=\textbf{r}^q(u_i,v_j),\quad i,j=0,\cdots,N-1.
\enn
Given a density $\varphi$, it can then be approximated on $\Gamma_q$ by the Chebyshev polynomials as
\ben
\varphi (x) \approx \sum\limits_{i,j = 0}^{N-1} \varphi _{ij}^q a_{ij}(u,v)=\sum\limits_{i,j = 0}^{N-1} \varphi \left(x_{ij}^q\right) a_{ij}(u,v), \quad x\in \Gamma_q,
\enn
where
\ben
a_{ij}(u,v)=\frac{1}{N^2}\sum^{N-1}_{m,n=0} \alpha_n\alpha_mT_n(u_i)T_m(v_j)T_n(u)T_m(v),\quad {\alpha _n} = \begin{cases}
	1, & n=0,\cr
	2, & n\neq 0.
\end{cases}
\enn

The strategy for evaluating the integral $\mathcal H_q(x)$ depends on the distance between the point $x$ and the $q$-th patch $\Gamma_q$. Define the distance
\ben
{\rm dist}_{x,\Gamma_q}:=\mathop{{\rm min}}\limits_{(u,v)\in[-1,1]^2}\{|x-\textbf{r}^q(u,v)|\}
\enn
and the value
\ben
(\widetilde{u}^q,\widetilde{v}^q)=\mathop{{\rm arg}{\rm min}}\limits_{(u,v)\in [-1,1]^2}\{|x-\textbf{r}^q(u,v)|\}.
\enn
Let $\delta>0$ be a given tolerance, where in our settings, $\delta=0.1$. If ${\rm dist}_{x,\Gamma_q}>\delta$, called ``non-adjacent" integration case, the integral is smooth, which can be accurately by means of Fej\'er's first quadrature rule. For the ``adjacent" integration case (i.e., ${\rm dist}_{x,\Gamma_q}\le\delta$), the integrals $\mathcal{H}_q(x)$ will be nearly/completely singular. In order to tackle this difficulty, we construct a new graded mesh, by means of appropriate ``rectangular-polar" changes of variables
\ben
\xi_\alpha(t)=\begin{cases}
		{\rm sgn}(t)-({\rm sgn}(t)-\alpha)\chi_p(1-|t|), &\alpha\ne\pm 1,\cr
		-1+2\chi_p(\frac{|t+1|}{2}),&\alpha=1,\cr
		1-2\chi_p(\frac{|t-1|}{2}),&\alpha=-1,
	\end{cases}\quad t\in[-1,1],
\enn
where$-1\le\alpha\le1$, and for a given integer $p$, $\chi_p$ is given by
\ben
\chi_p(s)=2\frac{[\eta_p(s)]^p}{[\eta_p(s)]^p+[\eta_p(-s)]^p}-1,\quad -1\le s\le 1,
\enn
with
\ben
\eta_p(s)=\left(\frac{1}{2}-\frac{1}{p}\right)s^3+\frac{1}{p}s+\frac{1}{2}.
\enn
The changes of variables, with both $\alpha=\widetilde{u}^q$ and $\alpha=\widetilde{v}^q$ can be used to produce a refinement around the points $\widetilde{u}^q$ and $\widetilde{v}^q$. Then applying the Chebyshev expansion of $\varphi$, we have
\ben
\mathcal{H}_q(x)\approx \sum\limits_{m,n=0}^{N-1} A^{q}_{nm} \varphi_{nm}^q,
\enn
where
\ben
&&A_{nm}^{q}=\sum\limits_{l_1.l_2=0}^{N_\beta-1}H(x,\textbf{r}^q(\xi_{\widetilde{u}^q}(\widetilde t_{l_1}),\xi_{\widetilde{v}^q}(\widetilde t_{l_2})))J^q(\xi_{\widetilde{u}^q}(\widetilde t_{l_1}),\xi_{\widetilde{v}^q}(\widetilde t_{l_2}))\\
&&\qquad \qquad \qquad\quad\times a_{nm}(\xi_{\widetilde{u}^q}(\widetilde t_{l_1}),\xi_{\widetilde{v}^q}(\widetilde t_{l_2}))\xi_{\widetilde{u}^q}^{'}(\widetilde{t}_{l_1})\xi_{\widetilde{v}^q}^{'}(\widetilde{t}_{l_2})\widetilde{w}_{l_1}\widetilde{w}_{l_2},
\enn
with the quadrature nodes and weights given by
\ben
\widetilde t_j=\cos \left(\frac{2j+1}{2N_\beta}\pi\right),\quad j=0,...,N_\beta-1,
\enn
and
\ben
\widetilde w_j=\frac{2}{N_\beta}\left(1-2\sum\limits_{l=1}^{\left\lfloor N_\beta/2 \right\rfloor }\frac{1}{4l^2-1}\cos(l\widetilde t_j)\right), \quad j=0,...,N_\beta-1.
\enn

Finally, we describe the approximations of the tangential derivative operator $\frac{d}{ds}$ and $\nu\times\nabla$. Note that $\nu\times\nabla=\nu\times\nabla^S$ where $\nabla^S$ denotes the surface gradient. On each patch $\Gamma_q$, the quantities $\frac{d\varphi}{ds}$ and $\nabla^S\varphi$ can be easily evaluated by means of term-by-term differentiation of the Chebyshev expansion of $\varphi$. In two dimensions, we can obtain that
\ben
\frac{d\varphi}{ds}(\textbf{r}^q(u))=\sum\limits_{i=0}^{N-1}\left|\frac{d\textbf{r}^q(u)}{du} \right|^{-1} \frac{da_i(u)}{du}\varphi_i^q.
\enn
In three dimensions, the evaluation of the tangential-derivative operator $\nu\times\nabla^S$ can be achieved by using the expressions of $\nabla^S$ on parameterized curve, see~\cite[Section 4.2.2]{BY20}.

\begin{remark}
If corners or edges exist on $\Gamma^b$, then the unknowns in the integral equations have singularities at the corners or edges. To resolve the singularities, we utilize a change of variables whose derivatives vanish at corners or along edges. A change of variables can be devised on the basis of mappings $\chi_p(s)$, whose derivatives up to order $p-1$ vanish at the endpoints. Then the function $\chi_p(s)$ can be used to construct a change of variables to accurately resolve field singularities at the corners or edges (for more details, see~\cite{BG20,CK98}).
\end{remark}

\subsection{Layered-medium scattering problems}
\label{sec4.3}

\begin{figure}[htb]
	\centering
	\begin{tabular}{cc}
		\includegraphics[scale=0.2]{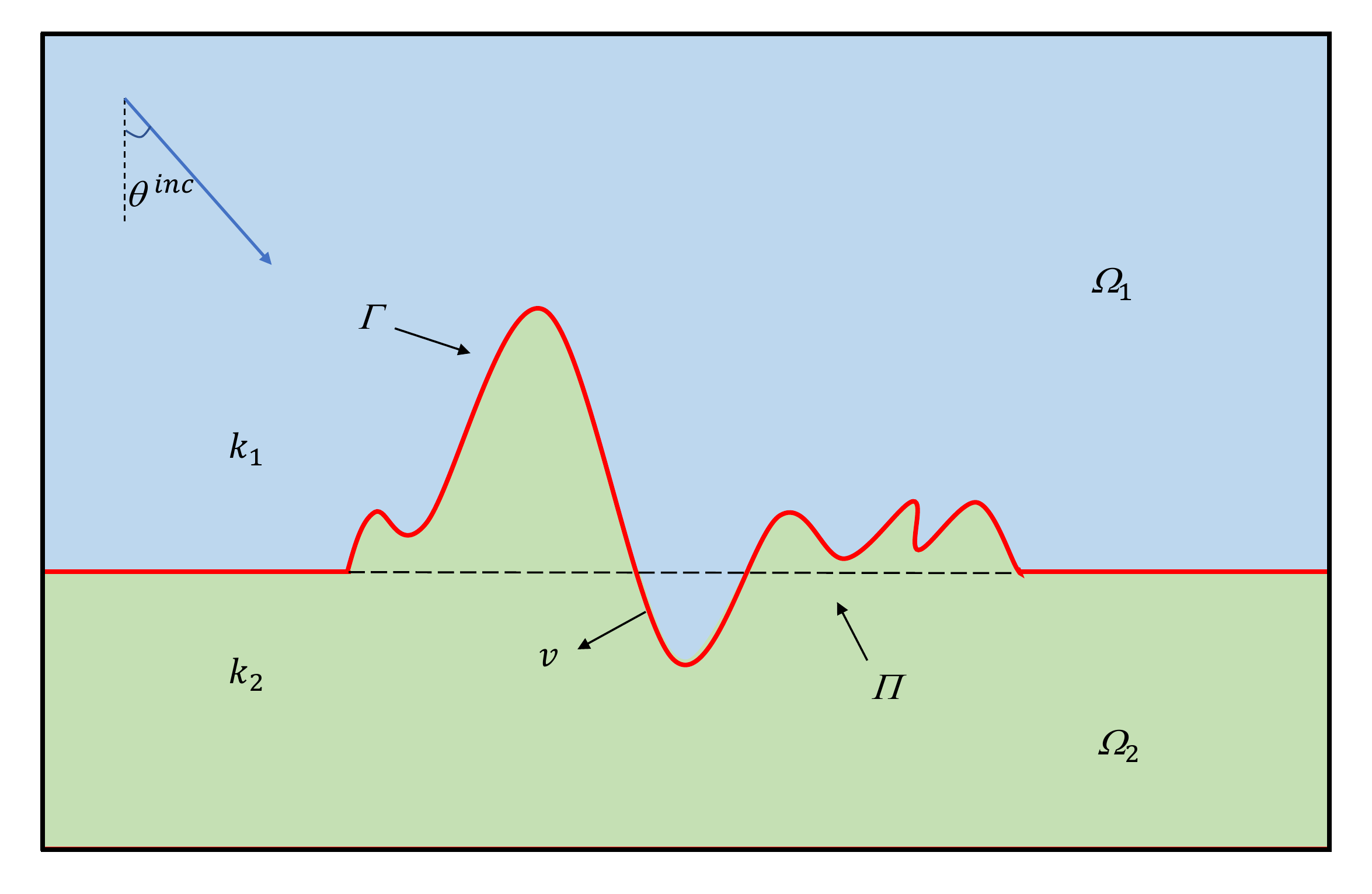}
	\end{tabular}
	\caption{Description of the problem under consideration: scattering by a defect on a penetrable layer. $\Gamma$ denotes the interface between the two media, and $\Phi$ denotes the interface between the upper and lower half-planes.}
	\label{TLP}
\end{figure}

To conclude this section, we briefly discuss how to the new PML-based BIE
solver be extended to layered-medium scattering problems. Without loss of generality, we focus on the acoustic scattering problems in a two-layer
medium in two dimensions, see Fig.~\ref{TLP}. The method proposed in~\cite{LLQ18} uses the Neumann-to-Dirichlet map to construct the
boundary integral equations. Alternatively, in this work, we utilize the second-kind system of integral equations derived in~\cite{KK75} which covers all four BIOs (\ref{SO})-(\ref{HO}).

Analogous to Section~\ref{sec2.1}, let $u^\mathrm{inc}$ be a plane incident wave
in $\Omega^1$ given in (\ref{PPW}) and $u_{j,\mathrm{ref}}^{\mathrm{sca}}$ be the reference scattered fields in $\Omega_j, j=1,2$, resulting from the
scattering of the plane wave $u^\mathrm{inc}$ by the flat surface $\Pi$; see~\cite{LLQ18} for more details. Then, the acoustic scattering problems can be formulated as follows: the scattered fields $u_j^{\rm sca}, j=1,2$ satisfy the Helmholtz equation
\ben
\Delta u_j^{\rm sca}+k_j^2u_j^{\rm sca}=0\quad\mbox{in}\quad \Omega_j,\quad j=1,2,
\enn
where $k_j$ is the wavenumber in $\Omega_j$, and the transmission conditions on $\Gamma$ are
\ben
&u^{\rm sca}_1|_\Gamma-u^{\rm sca}_2|_\Gamma= f (:=-u^{\rm inc}|_\Gamma-u_{1,\mathrm{ref}}^{\mathrm{sca}}|_\Gamma+u_{2,\mathrm{ref}}^{\mathrm{sca}}|_\Gamma),\\
&\partial_\nu u^{\rm sca}_1|_\Gamma-\partial_\nu u^{\rm sca}_2|_\Gamma= g
(:=-\partial_\nu u^{\rm inc}|_\Gamma-\partial_\nu u_{1,\mathrm{ref}}^{\mathrm{sca}}|_\Gamma+ \partial_\nu u_{2,\mathrm{ref}}^{\mathrm{sca}}|_\Gamma).
\enn
They lead to the following integral equation system
\ben
\begin{bmatrix}
	I+K_1-K_2& S_2-S_1\\
	N_1-N_2&I+K'_2-K'_1
\end{bmatrix}\begin{bmatrix}
	u^{\rm sca}_2|_\Gamma\\
	\partial_{\nu} u^{\rm sca}_2|_\Gamma
\end{bmatrix}=\begin{bmatrix}
	-f\\
	-g
\end{bmatrix},
\enn
where the subscripts $1$ and $2$ in the BIOs indicate that they are defined for the wave numbers $k_1$ and $k_2$, respectively. Applying the PML stretching and assuming that the corresponding solutions $\widetilde u^{\rm sca}_j, j=1,2$ to the PML truncated problems vanish on the outer boundary of PML region, we can obtain the reduced BIEs:
\ben
\begin{bmatrix}
	I+\widetilde K_1-\widetilde K_2&\widetilde S_2-\widetilde S_1\\
	\widetilde N_1-\widetilde N_2&I+\widetilde K'_2-\widetilde K'_1
\end{bmatrix}\begin{bmatrix}
	\widetilde u^{\rm sca}_2|_\Gamma\\
	\widetilde \partial_{\nu}\widetilde u^{\rm sca}_2|_\Gamma
\end{bmatrix}=\begin{bmatrix}
	-\widetilde f\\
	-\widetilde g
\end{bmatrix}.
\enn
Numerical schemes in the previous two subsections can be used to directly discretize the integral operators so as to obtain $\widetilde{u}_2^{\rm sca}|_{\Gamma}$ and $\widetilde \partial_{\nu}\widetilde u^{\rm sca}_2|_\Gamma$ numerically.

\begin{figure}[htb]
	\centering
	\begin{tabular}{cc}
		\includegraphics[scale=0.18]{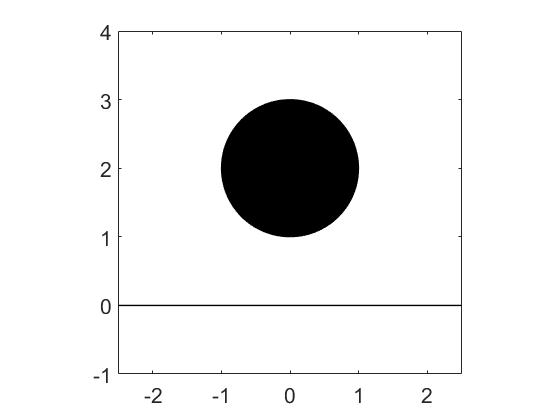} &
		\includegraphics[scale=0.18]{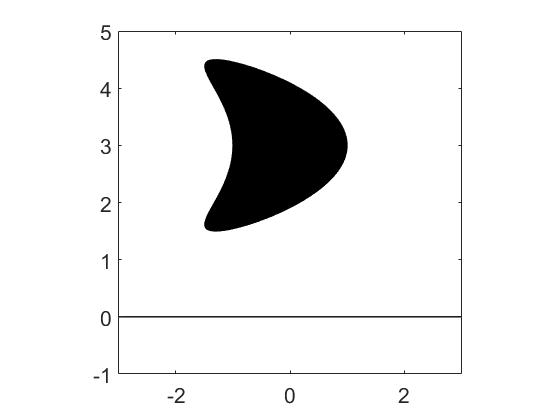} \\
		(a) disc within half-space & (b) kite within half-space
	\end{tabular}
	\caption{Two half-space in two dimensions considered in this paper.}
	\label{2Dob}
\end{figure}

\section{Numerical experiments}
\label{sec5}

In this section, we present a variety of numerical results to demonstrate the accuracy and efficiency of the proposed PML-based BIE methods for solving two- and three-dimensional acoustic half-space problems. Solutions of the various integral equations were produced by means of the fully complex version of the iterative solver GMRES with residual tolerance $\epsilon_r$ as specified in each case and the relative maximum error is defined by
\be
\label{RE}
\epsilon_{\infty}&:=\frac{\mbox{max}_{x\in \Gamma_{\mathrm{test}}}{\left|u^{\rm num}(x)-u^{\rm ref}(x)\right|}}{\mbox{max}_{x\in \Gamma_{\mathrm{test}}}{\left|u^{\rm ref}(x)\right|}},
\en
where $u^{\rm ref}$ is produced through evaluation of exact solutions, when available, or by means of numerical solution with sufficiently fine discretizations, and where $\Gamma_{\mathrm{test}}$ is a suitably selected line segment (2D) or square plane (3D) above the defect. In all cases, we choose $N^\beta=200$, $P=6$, $S=6$ and the PML thickness $T_i$ is set to be $T_i=2\lambda=2\times\frac{2\pi}{k}, i=1,\cdots,d$, \textcolor{r1}{where $\lambda$ denotes the free-space wavelength}. All of the numerical results presented in this paper were obtained by means of Fortran implementations, parallelized using OpenMP.

\begin{figure}[htb]
	\centering
	\begin{tabular}{cc}
		\includegraphics[scale=0.2]{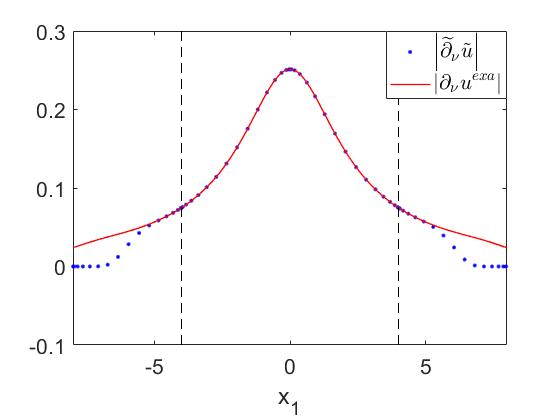} &
		\includegraphics[scale=0.2]{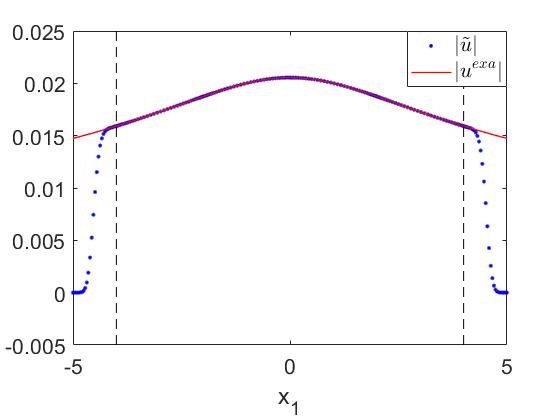} \\
		(a) Dirichlet problem ($k=\pi$) &(b) Neumann problem ($k=10\pi$)   \\
	\end{tabular}
	\caption{(a) Absolute values of $ \widetilde\partial_{\nu}\widetilde u^\mathrm{sca}$ and $\partial_\nu u^{\rm exa}$ on \textcolor{r1}{$\partial\Omega^b\cap\Pi$} for Dirichlet problem with $k=\pi$; (b) Absolute values of $ \widetilde u^\mathrm{sca}$ and $ u^{\rm exa}$ on \textcolor{r1}{$\partial\Omega^b\cap\Pi$} with $k=10\pi$; dashed lines separate \textcolor{r1}{$\partial\Omega_{\rm{PML}}\cap\Pi$} and \textcolor{r1}{$\partial\Omega_{\rm{PHY}}\cap\Pi$}.}
	\label{2DDN}
\end{figure}

\begin{figure}[htb]
	\centering
	\begin{tabular}{cc}
		\includegraphics[scale=0.2]{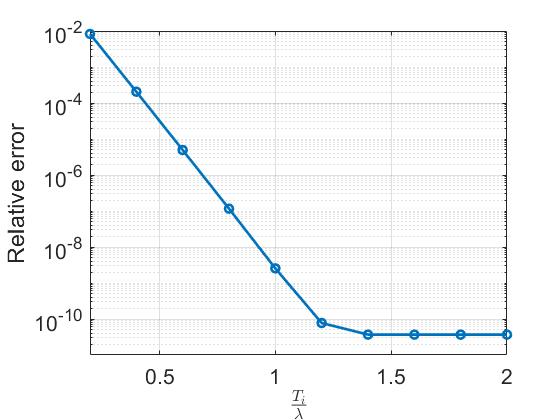} &
		\includegraphics[scale=0.2]{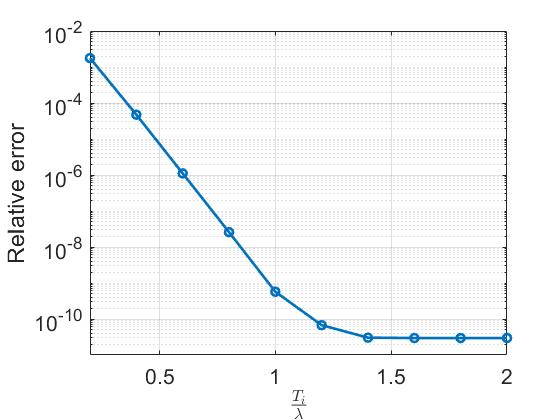} \\
		(a) Dirichlet problem & (b) Neumann problem \\
	\end{tabular}
	\caption{Numerical errors $\epsilon_{\infty}$ with respect to $\frac{T_i}{\lambda}$}
	\label{RET2d}
\end{figure}

\begin{table}[htb]
	\centering
	\caption{Numerical errors $\epsilon_{\infty}$ for the Dirichlet and Neumann problems of scattering by a disc-shaped obstacle within a half-space.}
	\begin{tabular}{|c|c|c|c|c|c|c|}
		\hline
		\multirow{2}{*}{$k$}&\multirow{2}{*}{$N$} & \multirow{2}{*}{$N_{\rm DOF}$}  & \multicolumn{2}{|c|}{Dirichlet problem}& \multicolumn{2}{|c|}{Neumann problem} \\
		\cline{4-7}
		&& & $N_{iter}$ & $\epsilon_{\infty}$ & $N_{iter}$&$\epsilon_{\infty}$ \\
		\hline
		&16&$6\times16$&16&$3.92\times10^{-3}$&12&$5.42\times 10^{-5}$\\
   $\pi$&32&$6\times32$&13&$1.61\times10^{-6}$&12&$2.66\times 10^{-8}$\\
		&64&$6\times64$&13&$2.29\times10^{-11}$&12&$7.64\times 10^{-13}$\\
		\hline
		&32&$12\times32$&48&$5.81\times 10^{-3}$&46&$2.17\times 10^{-4}$\\
 $10\pi$&48&$12\times48$&47&$5.10\times10^{-4}$&44&$1.63\times 10^{-5}$\\
		&64&$12\times64$&46&$4.69\times10^{-8}$&45&$1.93\times 10^{-9}$\\
		\hline
	\end{tabular}
\label{2Ddisc}
\end{table}

\begin{table}[htb]
	\centering
	\caption{Numerical errors $\epsilon_{\infty}$ for the Dirichlet and Neumann problems of scattering by a kite-shaped obstacle within a half-space.}
	\begin{tabular}{|c|c|c|c|c|c|c|}
		\hline
		\multirow{2}{*}{$k$}&\multirow{2}{*}{$N$} & \multirow{2}{*}{$N_{\rm DOF}$}  & \multicolumn{2}{|c|}{Dirichlet problem}& \multicolumn{2}{|c|}{Neumann problem} \\
		\cline{4-7}
		&& & $N_{iter}$ & $\epsilon_{\infty}$ & $N_{iter}$&$\epsilon_{\infty}$ \\
		\hline
		&16&$10\times16$&22&$3.84\times10^{-3}$&21&$6.65\times10^{-4}$ \\
   $\pi$&32&$10\times32$&22&$1.53\times10^{-6}$&21&$2.09\times10^{-7}$  \\
		&64&$10\times64$&22&$1.00\times10^{-10}$&21&$3.90\times10^{-11}$  \\
		\hline
        &32&$15\times32$&108&$1.69\times10^{-2}$&104&$1.30\times 10^{-3}$ \\
 $10\pi$&48&$15\times48$&108&$1.52\times10^{-4}$&104&$2.58\times 10^{-5}$ \\
		&64&$15\times64$&108&$5.49\times10^{-8}$&104&$3.31\times 10^{-9}$ \\
		\hline
	\end{tabular}
\label{2Dkite}
\end{table}

\begin{figure}[htb]
	\centering
	\begin{tabular}{cc}
		\includegraphics[scale=0.2]{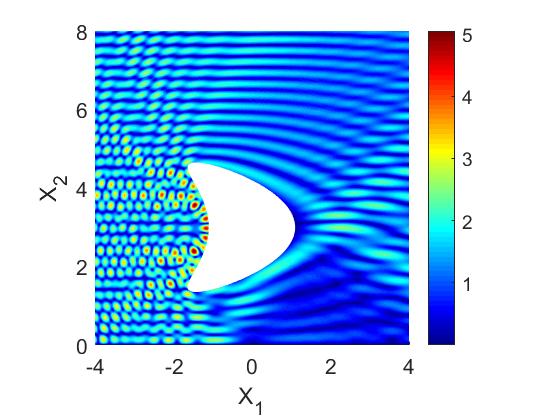} &
		\includegraphics[scale=0.2]{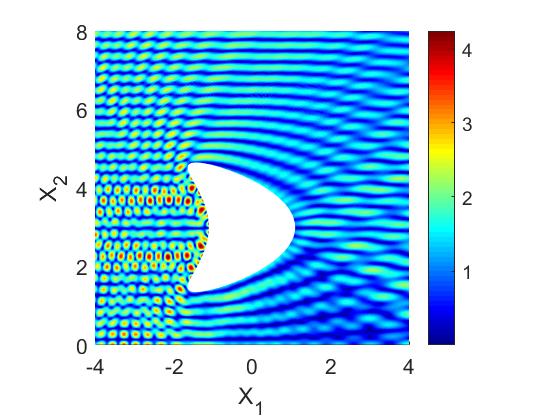} \\
		(a) Dirichlet problem &(b) Neumann problem \\
	\end{tabular}
	\caption{Absolute values of the total field for the Dirichlet and Neumann problems of scattering of a plane pressure wave by a kite-shaped obstacle, where $k=5\pi$.}
	\label{kitescat}
\end{figure}

\begin{figure}[htb]
	\centering
	\begin{tabular}{cc}
		\includegraphics[scale=0.2]{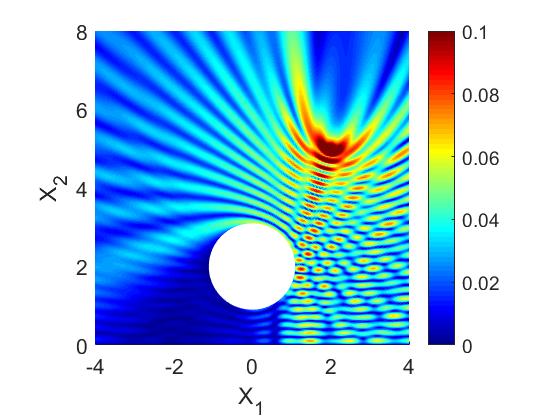} &
		\includegraphics[scale=0.2]{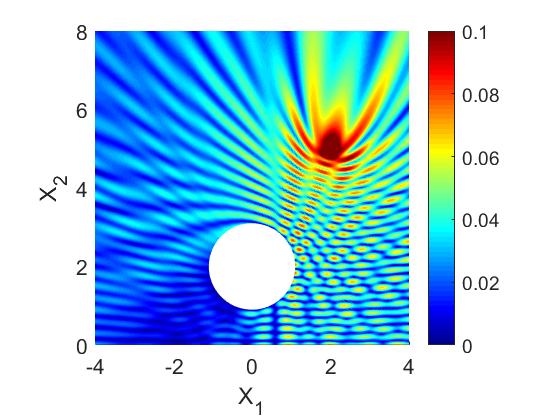} \\
		(a) Dirichlet problem &(b) Neumann problem \\
	\end{tabular}
	\caption{Absolute values of the total field for the Dirichlet and Neumann problems of scattering of an incident point source by a disc-shaped obstacle, where $k=5\pi$.}
	\label{discscat}
\end{figure}

\begin{figure}[htb]
	\centering
	\begin{tabular}{ccc}
		\includegraphics[scale=0.1]{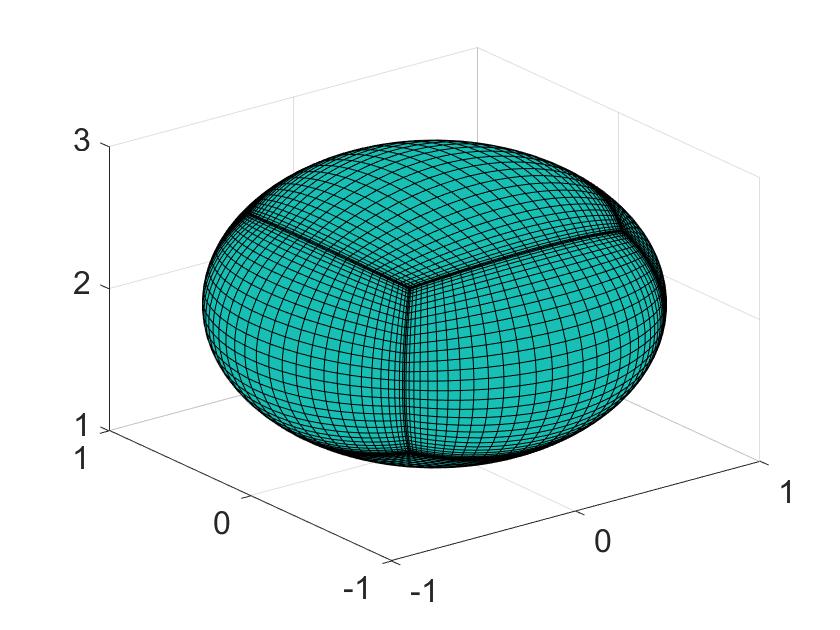} &
		\includegraphics[scale=0.1]{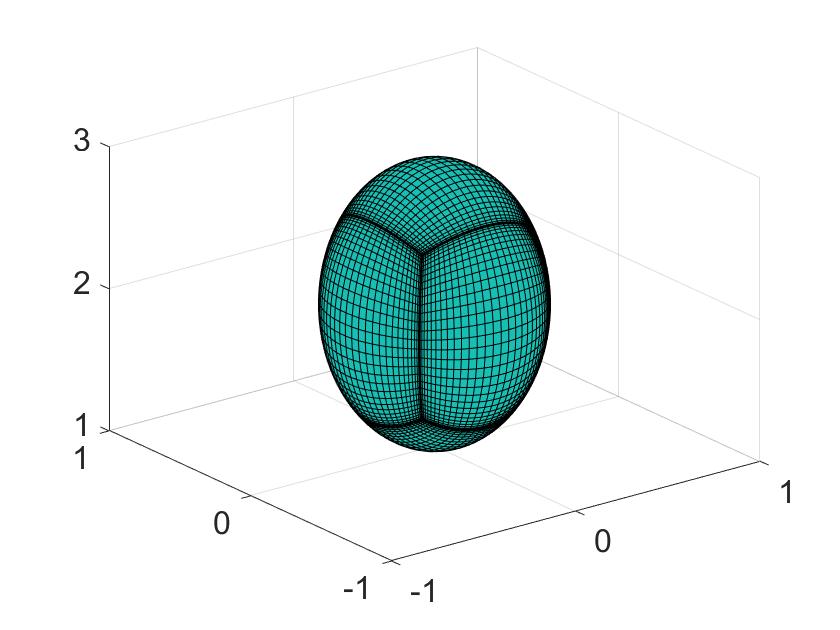} &
		\includegraphics[scale=0.1]{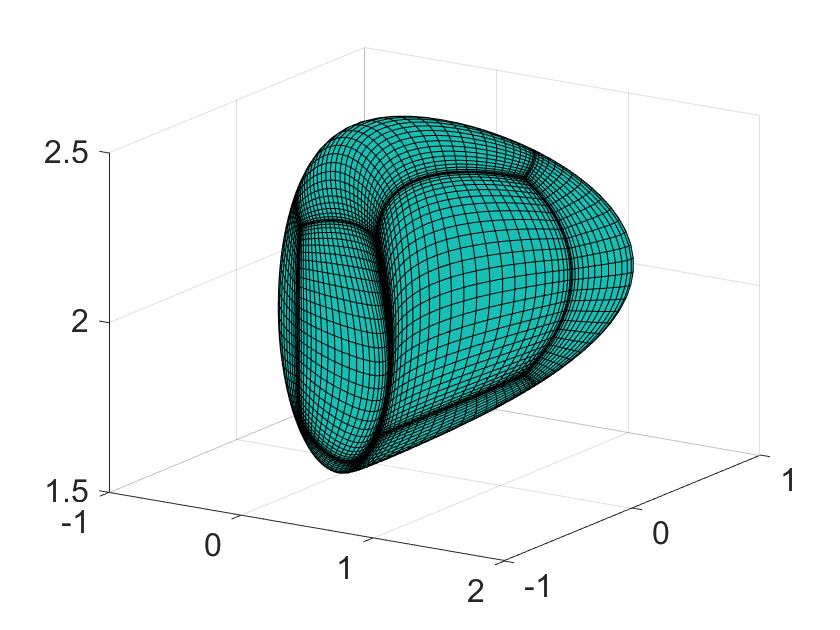} \\
		(a) Ball &(b) Ellipsoid&(c) Bean \\
	\end{tabular}
	\caption{Obstacles used in the numerical tests presented in Section~\ref{sec5.2}.}
	\label{3Dob}
\end{figure}

\begin{figure}[htb]
	\centering
	\begin{tabular}{cc}
		\includegraphics[scale=0.2]{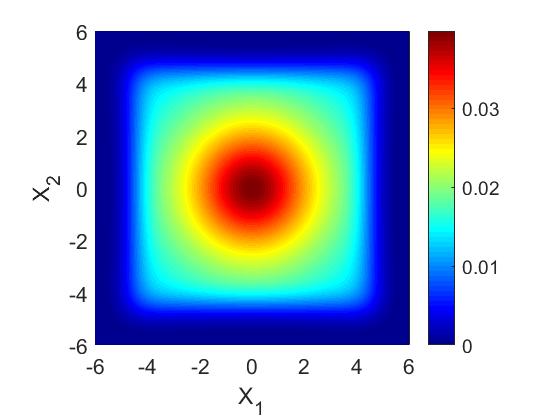} &
		\includegraphics[scale=0.2]{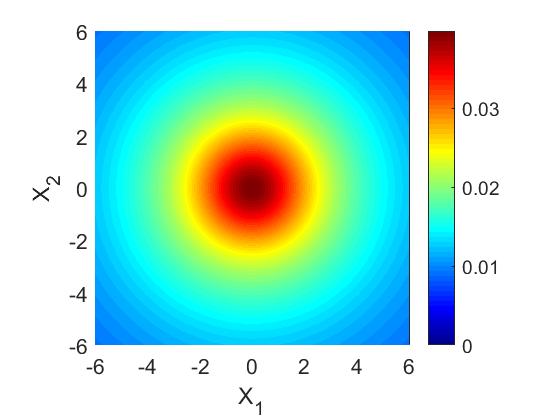} \\
		(a) $\left|\widetilde\partial_{\nu}\widetilde u^\mathrm{sca}\right|$ &(b) $\left|\partial_\nu u^{\rm exa}\right|$ \\
		\includegraphics[scale=0.2]{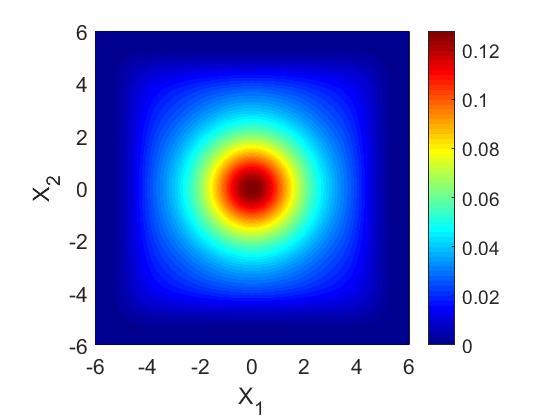} &
		\includegraphics[scale=0.2]{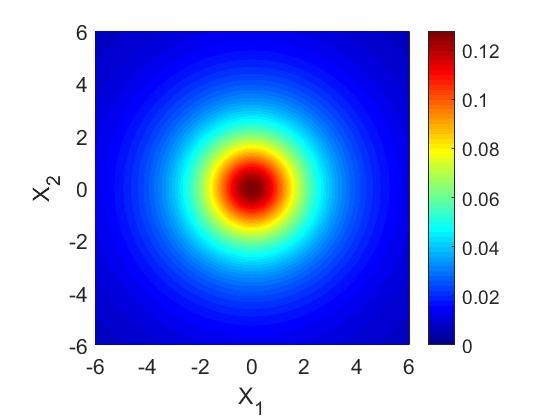} \\
		(c) $\left|\widetilde u^\mathrm{sca}\right|$ &(d) $\left|u^{\rm exa}\right|$ \\
	\end{tabular}
	\caption{(a)(b): Absolute values of the $\widetilde\partial_{\nu}\widetilde u^\mathrm{sca}$ and $\partial_\nu u^{\rm exa}$ for the Dirichlet problem; (c)(d): Absolute values of the $\widetilde u^\mathrm{sca}$ and $u^{\rm exa}$ for the Neumann problem. Here, $k=\pi$.}
	\label{3Dpi}
\end{figure}

\begin{figure}[htb]
	\centering
	\begin{tabular}{cc}
		\includegraphics[scale=0.2]{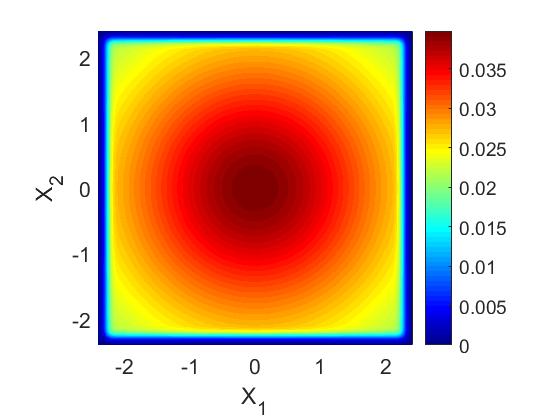} &
		\includegraphics[scale=0.2]{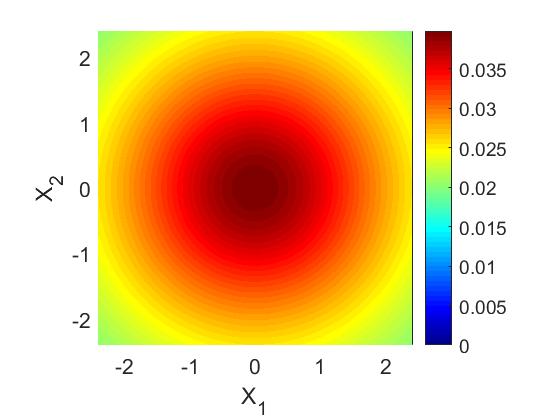} \\
		(a) $\left|\widetilde\partial_{\nu}\widetilde u^\mathrm{sca}\right|$ &(b) $\left|\partial_\nu u^{\rm exa}\right|$ \\
		\includegraphics[scale=0.2]{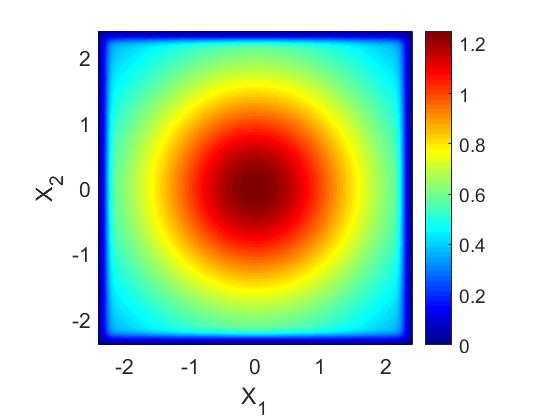} &
		\includegraphics[scale=0.2]{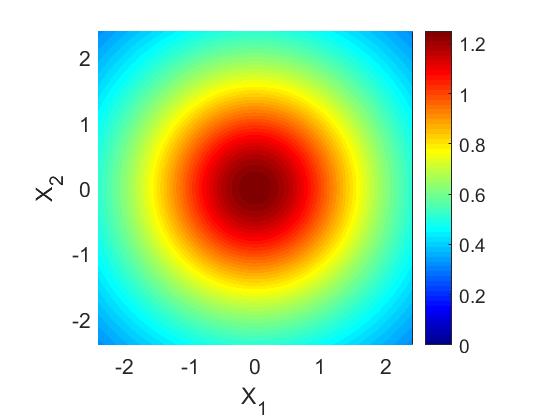} \\
		(c) $\left|\widetilde u^\mathrm{sca}\right|$ &(d) $\left|u^{\rm exa}\right|$ \\
	\end{tabular}
	\caption{(a)(b): Absolute values of the $\widetilde \partial_{\nu}\widetilde u^\mathrm{sca}$ and $\partial_\nu u^{\rm exa}$ for the Dirichlet problem; (c)(d): Absolute values of the $\widetilde u^\mathrm{sca}$ and $u^{\rm exa}$ for the Neumann problem. Here, $k=10\pi$.}
	\label{3D10pi}
\end{figure}

\begin{figure}[htb]
	\centering
	\begin{tabular}{cc}
		\includegraphics[scale=0.2]{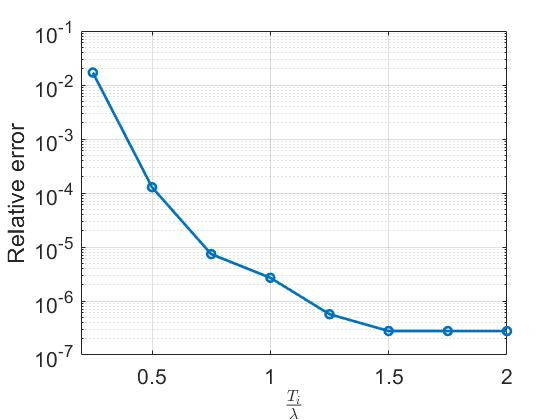} &
		\includegraphics[scale=0.2]{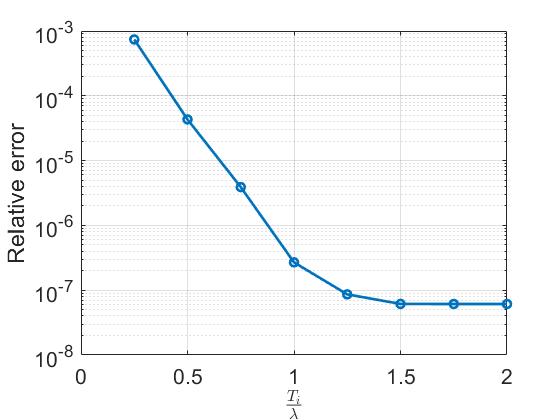} \\
		(a) Dirichlet problem & (b) Neumann problem \\
	\end{tabular}
	\caption{Numerical errors $\epsilon_{\infty}$ with respect to $\frac{T_i}{\lambda}$}
	\label{RET3d}
\end{figure}

\begin{figure}[htb]
	\centering
	\begin{tabular}{cc}
		\includegraphics[scale=0.2]{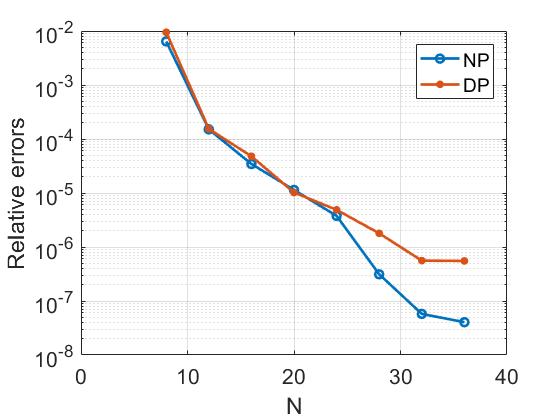} &
		\includegraphics[scale=0.2]{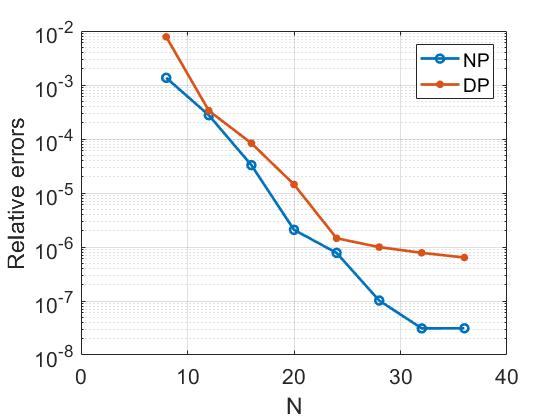} \\
		(a) Obstacle Fig.~\ref{3Dob}(a) & (b) Obstacle Fig.~\ref{3Dob}(b) \\
	\end{tabular}
	\caption{Numerical errors for the Dirichlet and Neumann problems of scattering by the obstacles given in Fig.~\ref{3Dob}(a-b) on the half-space. DP: Dirichlet problem, NP: Neumann problem}
	\label{SEPError}
\end{figure}

\begin{table}[htb]
	\centering
	\caption{Numerical errors and computing costs for the Dirichlet and Neumann problems of scattering by the obstacle Fig.~\ref{3Dob}(a) on the half-space.}
	\begin{tabular}{|c|c|c|c|c|c|}
		\hline
		\multirow{2}{*}{$\omega$} & \multirow{2}{*}{$N_{DOF}$}  & \multicolumn{4}{|c|}{Dirichlet problem} \\
		\cline{3-6}
		& & Time(prec.) & Time(1 iter.) & $N_{{iter}}$($\epsilon_r$)&$\epsilon_{\infty}$ \\
		\hline
		$\pi$&43008&2.02 min&8.23 s&$10(3.57\times10^{-10})$&$5.54\times10^{-7}$     \\
		$4\pi$&43008&1.17 min&5.14 s&21($2.76\times10^{-8}$)&$3.54\times10^{-6}$ \\
		$10\pi$&96768&5.11 min&34.26 s&40 ($5.73\times10^{-6}$)&$1.45\times10^{-5}$  \\
		\hline
		\multirow{2}{*}{$\omega$} & \multirow{2}{*}{$N_{DOF}$}&  \multicolumn{4}{|c|}{Neumann problem} \\
		\cline{3-6}
		&& Time(prec.) & Time(1 iter.) & $N_{{iter}}$($\epsilon_r$)&$\epsilon_{\infty}$ \\
		\hline
		$\pi$&43008 &1.65 min&7.57 s&9($9.45\times 10^{-10}$ )&$6.19\times10^{-8}$     \\
		$4\pi$&43008&1.14 min&7.24 s&20($8.47\times10^{-9}$)&$7.74\times10^{-7}$\\
		$10\pi$&96768&5.93 min&1.02 min&48($8.17\times10^{-9}$)&$8.92\times10^{-7}$ \\
		\hline
	\end{tabular}
	\label{cost}
\end{table}

\begin{figure}[htb]
	\centering
	\begin{tabular}{cc}
		\includegraphics[scale=0.2]{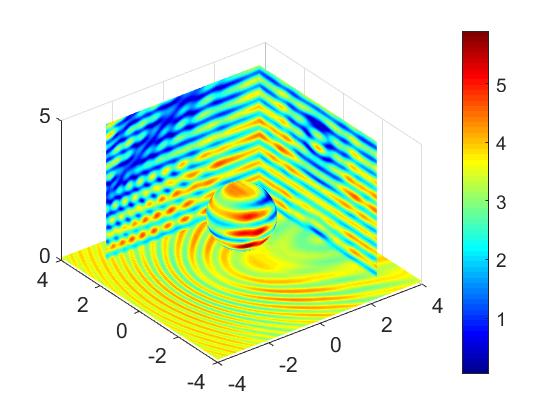} &
		\includegraphics[scale=0.2]{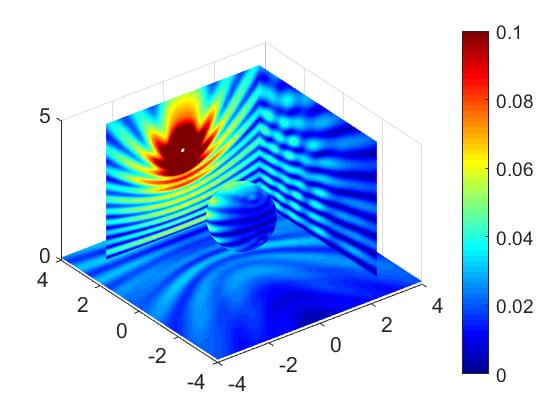} \\
		(a) Plane wave incidence &(b) Point source incidence \\
	\end{tabular}
	\caption{Absolute values of the total field resulting from the PML-based BIE method for the Neumann problem of scattering by the obstacle Fig.~\ref{3Dob}(a) on the half-space.}
	\label{SNP}
\end{figure}

\begin{figure}[htb]
	\centering
	\begin{tabular}{cc}
		\includegraphics[scale=0.2]{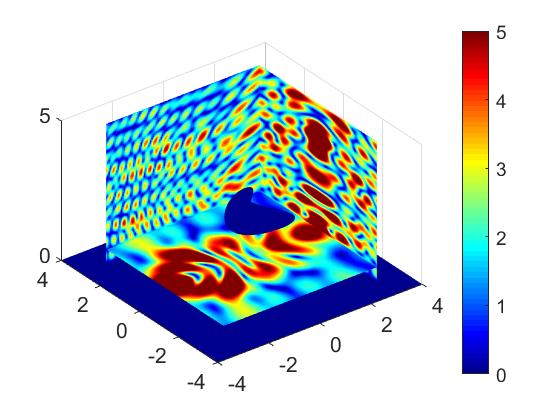} &
		\includegraphics[scale=0.2]{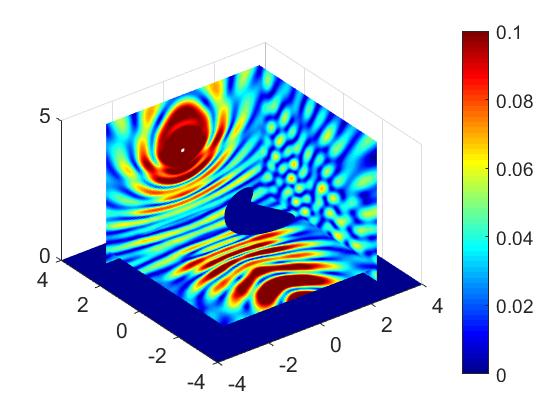} \\
		(a) Plane wave incidence &(b) Point source incidence \\
	\end{tabular}
	\caption{Absolute values of the total field resulting from the PML-based BIE method for the Dirichlet problem of scattering by the obstacle Fig.~\ref{3Dob}(c) on the half-space.}
	\label{BNP}
\end{figure}

\begin{figure}[htb]
	\centering
	\begin{tabular}{ccc}
		\includegraphics[scale=0.19]{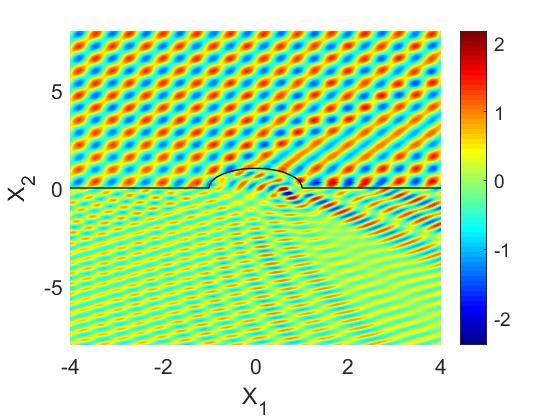} &
		\includegraphics[scale=0.19]{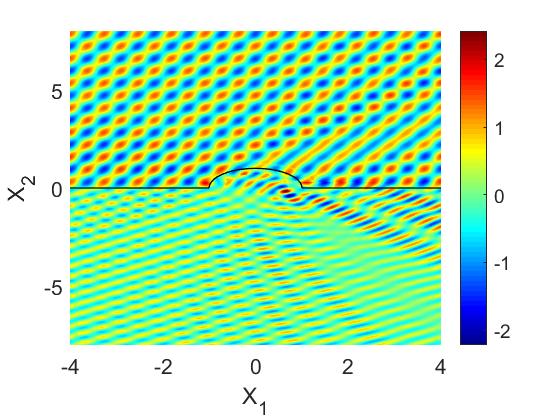} &
		\includegraphics[scale=0.19]{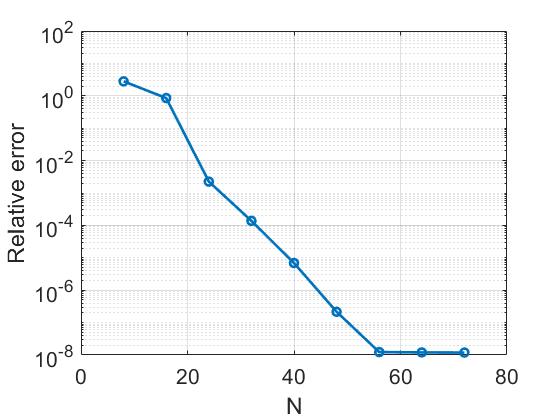} \\
		(a) Real parts &(b) Imaginary part &(c) Relative error $\epsilon_{\infty}$\\
	\end{tabular}
	\caption{(a,b) Real and imaginary parts of the total fields by the PML-based BIE method for the problem of scattering of a plane-wave by a semicircular bump; (c) The relative error $\epsilon_{\infty}$ of $u^{\rm tot}$ with respect to $N$.}
	\label{TLP2}
\end{figure}

\subsection{Two-dimensional examples}
\label{sec5.1}
In this subsection, we will test the accuracy and efficiency of the PML-based BIE methods for the two-dimensional problems with geometrical settings shown in Fig.~\ref{2Dob} wherein the thin black lines denote the impenetrable infinite boundary. To test the accuracy of the proposed solver, we consider the exact solution $u^{\rm exa}(x,z)=G(x,z)$ with $z=(0,2)$ for Fig.~\ref{2Dob}(a) and $z=(0,3)$ for Fig.~\ref{2Dob}(b), respectively, and the corresponding boundary data is given by $f=u^{\rm exa}(x,z)$ and $g=\partial_{\nu} u^{\rm exa}(x,z)$. We set $a_1=4$.

Firstly, we consider the Dirichlet problem of Fig.~\ref{2Dob}(a) with $k=\pi$ and the Neumann problem of Fig.~\ref{2Dob}(b) with $k=10\pi$. Fig.~\ref{2DDN} presents the absolute values of the numerical solutions of BIEs (\ref{PMLBIE1})-(\ref{PMLBIE2}) on $\Gamma^b$. It shows that the numerical solutions match perfectly with the exact data on $\partial\Omega_{\rm{PHY}}\cap\Pi$, and decay rapidly when approaching the endpoints of $\Gamma^b$. \textcolor{r1}{Next, we choose $k=2\pi$ and compute relative errors for different values of $\frac{T_i}{\lambda}$, as shown in Fig.~\ref{RET2d}. We can observe that the relative errors decay exponentially with the increase of PML thickness before the discretization error dominates the total error. It can also be seen that one can get sufficiently accurate solutions by choosing PMLs with a thickness of twice wavelength.} Tables~\ref{2Ddisc} and \ref{2Dkite} display the numerical errors $\epsilon_{\infty}$ for different degrees of freedom $N_{\rm DOF}$ as well as the corresponding numbers of iterations required by GMRES to achieve the residual tolerance $\epsilon_r=10^{-12}$. This clearly demonstrates the high accuracy of the proposed PML-based BIE solver. Next, we consider the scattering of the plane incident wave (\ref{PPW}) with $\theta^{\rm inc}=\frac{\pi}{4}$ and point source (\ref{PS}) located at $z=(2,5)\in\Omega$. Figs.~\ref{kitescat} and \ref{discscat} show the distribution of the total fields resulting from the PML-based BIE method for these two cases, respectively.

\subsection{Three-dimensional examples}
\label{sec5.2}

In this subsection, we demonstrate the performance of the proposed method for solving the three-dimensional scattering problems and the considered bounded obstacles on the half-space $\R^3_+$ are depicted in Fig.~\ref{3Dob}. Analogous to the discussion in two dimensions, we let $u^{\rm exa}(x,z)=G(x,z)$ be the exact solution, where $z=(0,0,2)$ locates inside the obstacles. We always set $a_1=a_2=2$.

In our first example, we consider the problems of Fig.~\ref{3Dob} (a) with $k=\pi$ and Fig.~\ref{3Dob} (b) with $k=10\pi$, and Fig.~\ref{3Dpi} and \ref{3D10pi} display the absolute values of the numerical solutions of BIEs (\ref{PMLBIE1})-(\ref{PMLBIE2}) on $\Gamma^b$, respectively. Here, 42 patches are selected for the boundary partitioning. It can be seen that the numerical solutions match perfectly with the exact data on $\partial\Omega_{\rm{PHY}}\cap\Pi$, and decay rapidly on PML interface $\partial\Omega_{\rm{PML}}\cap\Pi$. \textcolor{r1}{Choosing $k=2\pi$, the relative errors for different values of $\frac{T_i}{\lambda}$ are depicted in Fig.~\ref{RET3d}, which show that the relative errors decay exponentially with the increase of PML thickness and setting the PML thickness to twice the wavelength can obtain sufficiently accurate solutions.} Fig.~\ref{SEPError} presents the relative errors $\epsilon_{\infty}$ with respect to different $N$ (number of Chebyshev points in each patch) for the Dirichlet and Neumann problems with $k=\pi$, which clearly demonstrates the efficiency the proposed PML-based BIE solver for three-dimensional problems. Higher accuracy can be achieved by increasing the number $N^\beta$ and evaluating the integral kernels more carefully. Table~\ref{cost} lists the numerical errors together with other statistics such as pre-computation time, time per iteration and number of iterations used for the scattering problems at frequencies $k=\pi$, $k=4\pi$ and $k=10\pi$.

Next, we consider the scattering of the plane incident wave (\ref{PPW}) with angle pair $\theta^{\rm inc}=\frac{\pi}{3}$, $\phi^{\rm inc}=0$ and point source (\ref{PS}) located at $z=(0,3,3)\in\Omega$. Fig.~\ref{SNP} and \ref{BNP} show the distribution of the total fields with $k=4\pi$ resulting from the PML-based BIE method.

\section*{Acknowledgments}
WL is partially supported by NSF of Zhejiang Province for Distinguished Young Scholars Grant LR21A010001, NSFC Grant 12174310, and a Key Project of Joint Funds For Regional Innovation and Development (U21A20425). The work of LX is supported by NSFC Grant 12071060. TY gratefully acknowledges support from NSFC through Grant 12171465.

\end{document}